\documentclass[12pt,twoside,reqno]{amsart}

\usepackage{amssymb,amsmath,amstext,amsthm,amsfonts,amscd}

\usepackage[ansinew]{inputenc} 
\usepackage{graphicx}
\usepackage[mathscr]{eucal}
\usepackage{hyperref}
\usepackage{color}

\newcommand{\R}{\mathbb{R}}
\newcommand{\C}{\mathbb{C}}
\newcommand{\N}{\mathbb{N}}
\newcommand{\Z}{\mathbb{Z}}

\newcommand{\E}{\mathbb{E}}

\newcommand{\GL}{{\rm GL}}
\newcommand{\SL}{{\rm SL}}

\newcommand{\abs}[1]{\bigl| #1 \bigr|} 
\newcommand{\norm}[1]{{\left\lVert #1 \right\rVert}}

\newcommand{\ep}{\epsilon}

\newcommand{\dist}{{\rm dist}}

\newcommand{\Oscr}{\mathscr{O}}

\newcommand{\Lip}{\mathrm{Lip}}

\theoremstyle{plain}
\newtheorem{theorem}{Theorem}[section]
\newtheorem{proposition}{Proposition}[section]
\newtheorem{corollary}[proposition]{Corollary}
\newtheorem{lemma}[proposition]{Lemma}

\theoremstyle{definition}
\newtheorem{definition}{Definition}[section]

\theoremstyle{definition}
\newtheorem{remark}{Remark}[section]

\numberwithin{equation}{section}

\newcommand{\Pscr}{\mathcal{P}}
\newcommand{\Escr}{\mathcal{E}}
\newcommand{\Prob}{\mathrm{Prob}}
\newcommand{\Qop}{\mathcal{Q}}
\newcommand{\Pp}{\mathbb{P}}
\newcommand{\m}{m^-}
\newcommand{\Wsloc}{W^s_{\mathrm{loc}}}
\newcommand{\Wuloc}{W^u_{\mathrm{loc}}}

\newcommand{\muxminus}{\mu_{x^-}^u }
\newcommand{\muyminus}{\mu_{y^-}^u }	
\newcommand{\muxnminus}{\mu_{x_n^-}^u}	
\newcommand{\ind}{{\bf 1}}

\newcommand{\Xspace}{\mathfrak{X}}
\newcommand{\tilXsp}{\widetilde{\mathfrak{X}}}
\newcommand{\eu}{e_u}
\newcommand{\valfaSigma}{v_\alpha^X}

\def\real{\mathbb{R}}

\def\projective{\mathbb{P}}

\def\supp{\operatorname{supp}}
\def\dist{\operatorname{dist}}
\def\dim{\operatorname{dim}}
\def\id{\operatorname{id}}

\def\grass{\operatorname{Grass}}

\def\proj{\projective(\mathbb{R}^{d})}

\def\L1{L^1_{\mu}(M,N)}
\def\0{\hat{0}}
\def\SL{\operatorname{SL}(2,\real)}
\def\GL{\mathrm{GL}(d,\real)}

\newcommand{\lspan}{\mathrm{span}}

\title[H\"older continuity of the Lyapunov exponents]{H\"older continuity of the Lyapunov exponents of linear cocycles \\over hyperbolic maps}

\begin{document}

	\author[P. Duarte]{Pedro Duarte}
	\address{Departamento de Matem\'atica and CMAFcIO\\
		Faculdade de Ci\^encias\\
		Universidade de Lisboa\\
		Portugal 
	}
	\email{pmduarte@fc.ul.pt}
	
	\author[S. Klein]{Silvius Klein}
	\address{Departamento de Matem\'atica, Pontif\'icia Universidade Cat\'olica do Rio de Janeiro (PUC-Rio), Brazil}
	\email{silviusk@puc-rio.br}
	
	\author[M. Poletti]{Mauricio Poletti}
	\address{Universidade Federal do Cear\'a (UFC)}
	\email{mpoletti@mat.ufc.br}

\begin{abstract}   
Given a hyperbolic homeomorphism on a compact metric space, consider the space of linear cocycles over this base dynamics which are H\"older continuous and whose projective actions are partially hyperbolic dynamical systems. We prove that locally near any typical cocycle, the Lyapunov exponents are H\"older continuous functions relative to the uniform topology. This result is obtained as a consequence of a uniform large deviations type estimate in the space of cocycles. 
As a byproduct of our approach, we also establish other statistical properties for the iterates of such cocycles, namely a central limit theorem and a large deviations principle. 
\end{abstract}

	\maketitle


	\section{introduction}
	\label{intro}
	Let $M$ be a compact metric space with no isolated points and let $f \colon M\to M$ be a hyperbolic homeomorphism. Examples of such systems are Anosov diffeomorphisms on a compact manifold,  nontrivial hyperbolic attractors, horseshoes and Markov shifts. Moreover, Bowen~\cite{Bowen} showed that every hyperbolic homeomorphism is conjugated, via a H\"older continuous function to a topological Markov shift in a finite number of symbols.
	
	Given any H\"older continuous observable (referred to as a potential) on $M$, there exists an equilibrium state (which is unique, if we also assume the topological transitivity of the system), that is, there is an $f$-invariant Borel probability measure $\mu$ on $M$ which maximizes the pressure. Such measures, which are ergodic, correspond, via the semi-conjugation given by Bowen's theorem,  to measures that admit a local product structure with H\"older continuous density.   
	
	A linear cocycle over the base dynamical system $(M, f, \mu)$ is a skew product map 
	$$F_A \colon M \times \R^d \to M \times \R^d , \quad F_A (x, v) = \left(f (x), A(x) v\right) \, , $$
	where $A \colon M \to \GL$ is a H\"older continuous matrix valued function.
	
	The iterates of this new dynamical system are
	$$F_A^n (x, v) = \left( f^n (x), A^n (x) v \right) \, ,$$
	where
	$$A^n (x) := A (f^{n-1} x) \ldots A (f (x)) \, A (x) \, .$$

	A linear cocycle $F_A$  is determined by, and thus can be identified with the matrix valued, H\"older continuous function $A \colon M \to \GL$. We endow the space $ C^\alpha (M, \GL)$ of such functions with the uniform distance
	$$d(A,B):= \norm{A-B}_\infty + \norm{A^{-1}-B^{-1}}_\infty .$$
	
	By Furstenberg-Kesten's theorem, we have the following $\mu$-a.e. convergence of the maximal expansion of the iterates of the linear cocycle:
	\begin{equation} \label{F-K}
	\frac{1}{n} \, \log \norm{A^n (x)} \to L_1 (A) \, ,
	\end{equation}
	where the limit $L_1 (A)$ is called the maximal Lyapunov exponent of the cocycle $A$. 
	
	The other Lyapunov exponents are defined similarly: $L_2 (A)$ corresponds to the second largest expansion (or singular value) of the iterates of $A$ and so on, until $L_d (A)$.

	A linear cocycle $F_A$ induces the projective cocycle
	$$\hat{F}_A \colon M \times \proj \to M \times \proj , \quad F_A (x, \hat{v}) = \left(f (x), \widehat{A(x) v} \right) \, .$$
	
	We will assume that the projective cocycle $\hat{F}_A$ is a partially hyperbolic dynamical system, which is an open property  and corresponds, via the semi-conjugacy in Bowen's theorem, to the linear cocycle being fiber bunched (or nearly conformal). 
	
	We also assume that the cocycle $A$ is typical in the sense of Bonatti and Viana, which is an open and dense property and ensures the simplicity of the Lyapunov exponents.
	Precise definitions of this and other relevant concepts will be given in the next section.

	\medskip
	
	The iterates of a linear cocycle are multiplicative processes that generalize products of i.i.d. random matrices, which in turn represent multiplicative analogues of sums of i.i.d. scalar random variables. The study of their statistical properties, that is, of the convergence in~\eqref{F-K}, is an interesting problem in itself, with important consequences elsewhere.
	Limit theorems (e.g. a large deviations principle and a central limit theorem) were first obtained by E. Le Page~\cite{LP} for Bernoulli base dynamics and by P. Bougerol~\cite{Bou} for Markov type cocycles. Related results, in the same setting, were more recently established by P. Duarte and S. Klein, see~\cite{DK-Holder} and~\cite[Chapter 5]{DK-book}).
	
	For the case of linear cocycles over hyperbolic systems, in the same setting of this paper, S. Gou\"ezel and L. Stoyanov~\cite{GouS} obtained a large deviations principle, while K. Park and M. Piraino~\cite{park2020transfer} obtained a central limit theorem (and a large deviations principle).
	
	\smallskip
	
	We are interested in large deviations type (LDT) estimates that are finitary (non asymptotic), effective and uniform in the cocycle. Such estimates are also called concentration inequalities in probabilities.
	
	\begin{definition}
		Let $A \colon M \to \GL$ be a linear cocycle over an ergodic system $(M, f, \mu)$ as above. We say that $A$ satisfies an LDT estimate if there are constants $C=C(A)<\infty$ and $k=k(A)>0$   
		such that for all\, $0<\varepsilon<1$ and $n\in\N$, 
		$$  \mu \left\{ x\in M \colon \, \abs{\frac{1}{n}\,\log \norm{A^{n}(x)} - L_1(A) } > \varepsilon \,\right\} \leq C\,e^{ - {k\,\varepsilon^2 }\, n  } \;. $$
		
		We call such an estimate uniform if it holds for all cocycles in some neighborhood of $A$, with the same constants $C$ and $k$.
	\end{definition}

	The first result of this paper is the following.
	
	\begin{theorem}\label{LDT thm}
		Let $f \colon M\to M$ be a hyperbolic homeomorphism, let $\mu$ be an equilibrium state of a H\"older continuous potential and let $A\colon M \to \GL$ be a H\"older continuous linear cocycle. Assume that $A$ is typical and that the corresponding projective cocycle is a partially hyperbolic system. Then $A$ satisfies a uniform large deviations type estimate.
	\end{theorem}
	
	Ou approach is based upon the study of the spectral properties of the Markov (or transition) operator associated to the projective cocycle $\hat{F}_A$ and defined on an appropriate space of observables. In other related works, e.g.~\cite{park2020transfer}, it is the transfer operator that plays a similar r\^ole. Our method allows for a more quantitative control of various parameters, thus ensuring the uniformity of the LDT estimate above. 
	Moreover, as  a by-product of this approach and using an abstract CLT for stationary Markov processes due to Gordin and Lif\v{s}ic~\cite{Go69}, we establish the following central limit theorem in the setting of Theorem~\ref{LDT thm}.
	
	\begin{theorem}\label{clt intro}
		Given a cocycle $A$ as above, there exists $0<\sigma<\infty$ such that for every $v\in \R^d\setminus\{0\}$ and  $a\in\R$,
		$$
		\lim_{n\to \infty}  \mu\left\{ x\in M \colon  \frac{\log\norm{A^n(x) \, v }-n\, L_1(A)}{\sqrt{n}} \leq a \,  \right\} =  \frac{1}{\sigma\sqrt{2\pi}}\int_{-\infty}^a e^{-\frac{t^2}{2\sigma^2}}dt
		$$
	\end{theorem}
	
	Note that compared with the main result in~\cite{park2020transfer}, the positivity of the variance is an implicit conclusion rather than an additional hypothesis. 
	
	Furthermore, we also establish a large deviations principle.
	
	An important question in the theory of linear cocycles is the behavior of the Lyapunov exponents under small perturbations of the data. In particular, the continuity of the maximal Lyapunov exponent is considered a difficult problem in most settings. It has been studied successfully in the case of cocycles over Bernoulli systems by, amongst others, C. Bocker and M. Viana~\cite{BV}, E. Malheiro and M. Viana~\cite{Viana-M} and by L. Backes, A. Brown and C. Butler~\cite{BBB}.
	
	Furthermore, the study of finer continuity properties (e.g. H\"older continuity) of the Lyapunov exponents was initiated by Le Page~\cite{LePage} for the Bernoulli case and further extended to related settings, see for instance P. Duarte and S. Klein~\cite{DK-book, DK-Holder} and E. Tall and M. Viana~\cite{Tall-Viana}. 
	
	In~\cite[Chapter 2]{DK-book} it was established a relationship between the availability of {\em uniform} LDT estimates in any abstract space of cocycles and the H\"older continuity of the Lyapunov exponents. The main result of our paper is thus the following.
	
	\begin{theorem}\label{main A}
		Let $f \colon M \to M$ be a hyperbolic homeomorphism and let $\mu$ be an equilibrium state of a H\"older continuous potential. Consider the open set of typical H\"older continuous cocycles  $A\colon M \to \GL$ whose projective actions are partially hyperbolic. Then the Lyapunov exponents are locally H\"older continuous functions of the cocycle.
	\end{theorem}

	An important class of examples of linear cocycles are the Schr\"odinger cocycles, whose iterates represent the transfer matrices used to formally solve a discrete Schr\"odinger equation. More precisely, let $v \colon M \to \R$ be a potential function and let $\lambda > 0$ be a coupling constant. Given a base point $x \in M$,  the discrete Schr\"odinger operator $H_\lambda (x)$  acts on $l^2 (\Z)$ as follows: if  $\psi = \{\psi_n\}_{n\in\Z} \in l^2 (\Z)$, then
	$$\big[ H_\lambda (x) \psi \big]_n  := - (\psi_{n+1} + \psi_{n-1}) + \lambda \, v (f^n x) \psi_n \quad \text{ for all } n \in \Z \, .$$
	
	The associated  Schr\"odinger equation, $H_\lambda (x) \psi = E \psi$, where $E\in \R$, is a second order finite differences equation, which is solved recursively by iterating the cocycle 
	$$S_{E, \lambda} \colon M \to \SL  , \quad S_{E, \lambda} (x)  := \left[ \begin{array}{ccc} 
	\lambda v (x)  - E  & &  -1  \\
	1 & &  \phantom{-}0 \\  \end{array} \right] \, .$$
	
	Assuming that the base dynamics is an Anosov linear toral automorphism, for a smooth potential function $v$ and for small enough coupling constants $\lambda$, V. Chulaevski and T. Spencer~\cite{CS95} proved the positivity of the Lyapunov exponent of the cocycle $(f, S_{E, \lambda})$ for all energies, while J. Bourgain and W. Schlag~\cite{BS00} proved that the corresponding Schr\"odinger operator satisfies Anderson localization.
	
	In the (more general) setting of this paper we obtain the following. 
	
	\begin{theorem}\label{MP intro}  Assume that the base dynamics $(M, f, \mu)$ is given by a topologically mixing, uniformly hyperbolic $C^1$ diffeomorphism and that $v \colon M \to \R$ is a H\"older continuous potential function with  $\int v\, d\mu \neq 0$. 
		
		Given $\delta>0$ 
		there exists $\lambda_0>0$ such that for all $0<\abs{\lambda}<\lambda_0$ and $E\in\R$ with $\abs{E}<2-\delta$,  the cocycle
		$(f, S_{E,\lambda})$ has positive Lyapunov exponent.
		
		Moreover, the maximal Lyapunov exponent is a H\"older continuous function of the parameters $E$ and $\lambda$.
		
		Furthermore, the statistical properties derived above (uniform LDT estimates, CLT) also hold for the iterates of this cocycle.
	\end{theorem}
	
	We remark that positivity and H\"older continuity  of the Lyapunov exponents together with uniform LDT estimates for such cocycles are often used as main ingredients in establishing Anderson localization of the corresponding Schr\"odinger operator (see for instance J. Bourgain and M. Goldstein~\cite{B-G-first},  J. Bourgain and W. Schlag~\cite{BS00}, S. Jitomirskaya and X. Zhu~\cite{Jit-AL-random} or the survey of D. Damanik~\cite{David-survey}).  This topic, in the setting of this paper,  will be considered in a separate project.
	
	We would like to note that, according to an old private conversation between the second author and D. Damanik,  the Schr\"odinger operator with potential defined over similar types of base dynamics as in this paper has also been studied by A. Avila and D. Damanik. We are not aware of the current status of their project, nor of the methods employed, although we expect them to be different from ours.

	\smallskip
	
	The paper is organized as follows. In Section~\ref{intro} we formally define the main concepts used in this paper and provide the precise formulation of the main result, the continuity Theorem~\ref{main A}.  In Sections~\ref{holonomy} and~\ref{quasicompact} we introduce some technical tools such as the holonomy reduction, the Markov operator (and its quasi-compactness property) and stationary measures. In Sections~\ref{basemixing} and~\ref{fibermixing} we establish the strong mixing of the Markov operator associated to the base dynamics and respectively to the projective fiber dynamics. Based on this, in Section~\ref{continuity} we prove the uniform LDT estimate in Theorem~\ref{LDT thm} and as a consequence of an abstract continuity result, we derive the main Theorem~\ref{main A}. Moreover, as explained in Remark~\ref{LDP}, the argument can be adapted to also obtain a large deviations principle. Furthermore, the strong mixing of the Markov operator corresponding to the fiber dynamics together with an abstract central limit theorem are used in Section~\ref{clt section} to derive Theorem~\ref{clt intro}. Finally, in Section~\ref{example} we apply these results to Schr\"odinger cocycles and establish Theorem~\ref{MP intro}.
	
	\smallskip
	

	\section{Main concepts}
	\label{statements}
	 
	
	In this section we formally introduce the concepts mentioned in the introduction and formulate the results which imply  Theorem~\ref{main A}.
	
	\subsection{ Base dynamics }  
	
	Let $M$ be a compact metric space with no isolated points and let $f \colon M \to M$ be a homeomorphism. For a point $x\in M$ and for $\ep > 0$ small we define the local $\ep$-stable and respectively $\ep$-unstable sets of $x$ by
	\begin{align*}
	W^{s}_{\ep}\left(x\right) & :=\left\{ x \colon d \left( f^k(y),f^k(x) \right) <\epsilon \, \text{ for all }k\geq 0\right\} \text{ and }\\
	W^{u}_{\ep}\left(x\right) & :=\left\{ x \colon  d \left(f^k(x),f^k(y) \right) <\epsilon \, \text{ for all }k\leq 0\right\} .
	\end{align*}

	Following M.Viana~\cite{Viana2008}, we call a homeomorphism $f$  {\em (uniformly) hyperbolic} if there are constants $C<\infty$, $\ep>0$, $\tau > 0$ and $\lambda \in (0,1)$ such that for all $x\in M$ we have
	\begin{enumerate}
		\item $d \left( f^n(y_1),f^n(y_2) \right) \le C \, \lambda^n d(y_1,y_2)$ \, for all
		$y_1, y_2 \in W^s_{\ep} (x)$ and $n \ge 0$;
		\item $d \left (f^{-n} (y_1), f^{-n}(y_2) \right) \le C \, \lambda^{n} d (y_1,y_2)$ \, for all
		$y_1, y_2 \in W^u_{\ep} (x)$ and $n \ge 0$;
		\item if $d (x, y)\le\tau$, then $W^u_{\ep}(x)$ and $W^s_{\ep}(y)$ intersect in a
		unique point, which is denoted by $[x, y]$ and which depends continuously on $(x, y)$.
	\end{enumerate}
	
	For this $\ep$, the sets $\Wsloc (x) := W^s_{\ep} (x)$ and $\Wuloc (x) := W^u_{\ep} (x)$ are referred to simply as local stable and unstable sets of $f$.
	
	See also J. Ombach~\cite{Ombach} for other characterizations of hyperbolicity.

	Typical examples of such dynamical systems are Anosov diffeomorphisms, Markov shifts, non trivial hyperbolic attractors and horseshoes.
	
	Recall that given a H\"older continuous potential $\varphi \colon M \to \R$, there is a corresponding equilibrium state measure on $M$, which is unique if $f$ is topologically transitive. 
	
	It was essentially established by R. Bowen~\cite{Bowen} (see also V. Baladi~\cite{Baladi91} and the review by V. Alekseev and M. Yakobson~\cite{AY81} for more general settings that include ours) that uniformly hyperbolic homeomorphisms are semi-conjugated to a Markov shift. This means that for every uniformly hyperbolic homeomorphism $f \colon M\to M$  there exists a topological Markov shift $T \colon X \to X$ and a H\"older continuous map $\pi \colon X\to M$ such that $\pi\circ T=f\circ \pi$. Moreover, we can restrict $\pi$ so that it becomes a homeomorphism on a set that has total measure  for every $T$-invariant measure $\mu$. Furthermore, we have a one to one correspondence between the equilibrium states of the Markov shift $(X, T)$ and those of $(M, f)$.
	
	Let us then consider a topological Markov shift $(X, T)$, where the phase space $X =\{1,\ldots, \ell\}^{\Z}$, the transformation $T \colon X\to X$ is the left  shift homeomorphism 
	$T x = T \{x_i\}_{i\in \Z}:=\{x_{i+1}\}_{i\in \Z}$ and  the distance on $X$ is given by
	$$d(x, y) := 2^{-\min\{i \colon x_i\neq y_i\text{ or }x_{-i}\neq y_{-i}\}} \, .$$
	
	Assume that  $\mu\in \Prob(X)$ is a $T$-invariant measure.
	Given the symbols $a_0,\ldots, a_k\in\{1,\ldots, \ell\}$ consider the corresponding  
	cylinder 
	$$ [a_0, \ldots, a_k]:=\{ x\in X\colon x_i=a_i\, \text{ for }\, 0\leq i\leq k  \}. $$
	
	Define  $X^-:=\{1,\ldots, \ell\}^{-\N\cup \{0\}}$ and $X^+:=\{1,\ldots, \ell\}^{\N}$, where $\N:=\{1,2,3,\ldots\}$. 
	Let $T_+:X^+\to X^+$
	be the non-invertible left shift and $T_-:X^-\to X^-$ be the non-invertible right shift.
	Denote by $P_\pm \colon X\to X^\pm$ the corresponding canonical projections.
	Given $x\in X$, define the \textit{local stable set}
	$$\Wsloc (x):=\{ y\in X\colon y_n=x_n\, \text{ for all } \, n\geq 0 \} , $$
	and  the \textit{local unstable set} 
	$$\Wuloc (x):=\{ y\in X\colon y_n=x_n\, \text{ for all } \, n\leq  0 \} . $$
	Note that
	$ \Wuloc(x)= P_-^{-1}(P_-(x))$ for all $x\in X$ 
	and likewise $\Wsloc (x)= P_+^{-1}(P_+(x))\cap [x_0]$ for all $x\in X$. Hence
	\begin{align*}
	&  \Wuloc (x)=\Wuloc (y) \quad  \Leftrightarrow\quad   P_-(x)=P_-(y) .\\
	& \Wsloc (x)=\Wsloc (y) \quad  \Leftrightarrow\quad   x_0=y_0\, \text{ and }\, P_+(x)=P_+(y) .
	\end{align*} 
	Therefore, locally inside a cylinder $[i]$ of size $1$,  we can make the identifications 
	\begin{align*}
	[i]\cap X^-  & \equiv X/\Wuloc \equiv \Wsloc (q) \\
	X^+  & \equiv X/\Wsloc \equiv \Wuloc (q)
	\end{align*}
	where the right-hand sides represent the local stable and unstable sets of
	a reference point $q\in [i]$.
	
	Denote by $\mu_{[i]}$ the restriction of the measure $\mu$ to the cylinder $[i]$. 
	Define $\mu^+_i:= (P_+)_\ast \mu_{[i]}$, which is a measure on $X^+= \Wuloc (q)$ and $\mu^-_i:= (P_-)_\ast \mu_{[i]}$, which is a measure on $[i]\cap X^- = \Wsloc (q)$. Consider the lipeomorphism (by-Lipschitz homeomorphism)
	$$ h \colon [i]\to \Wsloc (q)\times \Wuloc (q)\equiv [i]\cap X^-\times X^+$$ defined by $h(x):=(P_-(x),P_+(x))$.

	\begin{definition}
		\label{def local product measure}
		We say that $\mu$ has local product structure with H\"older density if there exists  $\rho:X^-\times X^+ \to (0,+\infty)$
		a H\"older continuous function  such that  for each $i\in \{1,\ldots, \ell\} $, \, 
		$h_\ast \mu_{[i]} =  \rho\vert_{[i]}\, (\mu^-_i\times \mu^+_i)$.
	\end{definition}

	\begin{remark}
		By redefining the symbols, we note that it is sufficient for the product structure to occur at a smaller scale, corresponding to cylinders with length $>1$. 
	\end{remark}

	As explained below, examples of such measures are the equilibrium states of H\"older continuous potentials.

	Let $\mu$ be a $T$-invariant measure on $X$. Then  $\mu^+ :=(P_+)_*\mu$ is a $T_+$-invariant measure on $X^+$. Let $J_{\mu^+}T_+$ be the Jacobian of $T_+$.
	
	\begin{theorem}\label{t.holder.prod.meas}
		If $J_{\mu^+}T_+>0$ and it is H\"older continuous, then $\mu$ has local product structure with H\"older continuous density.
	\end{theorem}
	\begin{proof}
		Given $x^+$ and $y^+$ in the same cylinder, define $h_{x^+,y^+} \colon P_+^{-1}(x^+)\to P_+^{-1}(y^+)$ by 
		$h_{x^+,y^+}(x)=y$, where $y_n=y^+_n$ for $n>0$ and $y_n=x_n$ for $n\leq 0$. 
		
		Using Rokhlin's disintegration theorem, there exists a disintegration $x^+\mapsto \mu^s_{x^+}$ of $\mu$ relative to the partition $\{P_+^{-1}(x^+),x^+\in X^+\}$. By~\cite[Lemma~2.6]{BoV04}, the measure $(h_{x^+,y^+})_* \mu^s_{x^+}$ is absolutely continuous with respect to $\mu^s_{y^+}$. Moreover by \cite[Lemma~2.4]{BoV04} the \textcolor{blue}corresponding Jacobian $J h_{x^+,y^+}$, where $ (h_{x^+,y^+})_* \mu^s_{x^+}   = Jh_{x^+,y^+} \, \mu^s_{y^+} $,  
		is given by $$Jh_{x^+,y^+}(x) = \lim_{n\to\infty}\frac{J_{\mu^+}T_+(x_n)}{J_{\mu^+}T_+(y_n)} \, ,$$   where $x_n=P^+(T^{-n}(x))$.
		
		Given a cylinder $[i]$, fix $x^+\in [i]$, we can write $\mu_i=\rho \, \mu^s_{x^+}\times \mu^+$ with $\rho(y^+,z^-)=Jh_{x^+,y^+}(z)$, where $z_n=y^+_n$ if $n>0$ and $z_n=z^-_n$ if $n\leq 0$. So we are left to prove that $Jh_{x^+,y^+}(x)$ is H\"older continuous.
		
		To see this observe that $Jh_{x^+,y^+}(x)=\lim_{n\to\infty}\frac{J_{\mu^+}T_+(x_n)}{J_{\mu^+}T_+(y_n)}$, then we can see $Jh$ as the holonomy of the linear cocycle $J_{\mu^+}T_+(x^+)$ taking values on $\real$, then by Theorem~\ref{t.Holder.holonomies} below, $Jh_{x^+,y^+}$ varies H\"older continuously  with respect to $(x^+, y^+)$.
	\end{proof}

	\begin{remark}
		The above result  shows that measures with H\"older continuous Jacobian for the one sided shift $T^+$ have local product structure. An equilibrium state  for $T$ is the lift of an equilibrium state for $T^+$ and they have H\"older continuous Jacobian if the potential is H\"older continuous, see~\cite{Bowen}. So equilibrium states of H\"older continuous potentials satisfy our hypothesis.
	\end{remark}

	\subsection{ Fiber dynamics }

	Let $A \colon X \to \GL$ be a linear cocycle over the shift $T$.  Given $\beta>0$ we say that $A$ admits \emph{$\beta$-H\"older stable holonomies} if there exist $C < \infty$ and a family of linear maps $\left\{ H^s_{x, y}\in \GL \colon x \in X, y \in \Wsloc (x) \right\}$  with the following properties: 
	\begin{enumerate}
		\item $H^s_{x, x}=\id$ and $H^s_{x, y}\circ H^s_{y, z}=H^s_{x, z} $ for $x\in X$, $y, z \in \Wsloc (x)$,
		\item $A(y) \circ H^s_{x, y}=H^s_{T(x),T(y)} \circ A(x) $ for $x \in X$, $y \in \Wsloc (x)$,
		\item $\norm{H^s_{x, y}-\id}\le C \dist(x, y)^\beta $, 
		\item $\norm{H^s_{x, y}-H^s_{x', y'}}\leq C \left( d (x, x')^\beta + d (y, y')^\beta \right)$ for $y \in  \Wsloc (x)$ and $y'\in  \Wsloc (x')$.
	\end{enumerate}

	Moreover, we say that $A$ admits \emph{$\beta$-H\"older unstable holonomies} if the inverse cocycle admits $\beta$-H\"older stable holonomies.
	
	We say that an $\alpha$-H\"older linear cocycle $A:X\to \GL$ is \emph{fiber bunched} if there exists $N \in \N$, such that 
	$$\norm{A^N(x)}\norm{A^N(x)^{-1}}2^{-N\alpha}<1 \quad \text{ for all } \ x\in X . $$ 
	
	Note that a cocycle is fiber bunched if and only if the corresponding projective cocycle is a partially hyperbolic dynamical system.
	
	If $A$ is fiber bunched then there exists $\beta>0$ such that $A$ admits $\beta$-H\"older stable and unstable holonomies (see Theorem~\ref{t.Holder.holonomies}) and $\beta$ can be taken uniformly in a $C^0$ neighborhood of $A$. Moreover, for every $*=s$ or $u$, $x \in X$, $y \in W^*_{\mathrm{loc}} (x)$, the map $A\mapsto H^{*,A}_{x, y}$ varies continuously in the uniform topology.
	
	From now on we fix the base dynamics $T \colon X \to X$ with an ergodic measure $\mu$  that has local product structure with H\"older density and we consider the space of $\alpha$-H\"older linear cocycles over this dynamics. We endow this space with the H\"older norm 
	$$
	\norm{A}_\alpha=\sup_{x\in X}\norm{A(x)}+\sup_{x\neq y}\frac{\norm{A(x)-A(y)}}{\norm{x-y}^\alpha}.
	$$
	
	We recall the following definition from \cite{BoV04}.
	\begin{definition}
		\label{pinching and twisting}
		We say that a cocycle $A \colon X \to\GL$ satisfies the pinching and twisting conditions if there exists a $q$-periodic point  $a=T^q(a)$  and an associated homoclinic point  $z\in \Wuloc(a)$ with $z'=T^l z\in \Wsloc(a)$, for some $l\in\N$, both $a$ and $z\in \supp(\mu)$, such that the transition  map $\psi_{a,z,z'}:\R^d\to\R^d$, defined by  $\psi_{a,z,z'} := H^s_{z', a}\, A^l(z)\, H^u_{a, z}$, satisfies:
		\begin{enumerate}
			\item[(p)] The eigenvalues of the matrix $A^q(a)$ have multiplicity one and different absolute values.   
			Consider (for the next clause) the corresponding  eigenvectors $e_1,\ldots,e_d$ of $A^q(a)$.
			
			\item[(t)] $\lspan \{\psi_{a,z,z'}^{\pm 1}(e_i)\colon i\in I\} + \lspan\{e_j \colon j\in J\}=\R^d$ for any $I,J\subset \{1,\dots,n\}$ such that $\# I+\# J=n$.
		\end{enumerate}
	\end{definition}
	
	Let $g$ be the matrix of $\psi_{a,z,z'}$ written in the base of eigenvectors of $A^q (a)$; the twisting condition is then equivalent to all the algebraic minors of $g$ being different from zero.
	
	A cocycle satisfying the pinching and twisting properties is also called $1$-typical. 
	
	\subsection{ Continuity of the Lyapunov exponents }  
	
	We can now present the precise formulation of the main result of this paper.

	\begin{theorem}\label{theorem.main}
		Let $T \colon X\to X$ be the left shift, let $\mu$ be an ergodic measure that has local product structure with H\"older density and let $A \colon X\to \GL$ be an $\alpha$-H\"older fiber bunched linear cocycle. If $A$ is $1$-typical then there exists a $C^0$ neighborhood of $A$ such that in this neighborhood, the maximal Lyapunov exponent is a H\"older continuous function of the cocycle.
	\end{theorem}
	
	Given a linear cocycle $A \colon X \to \GL$ and an integer $1\le k \le d$, let $\wedge^k A$ be the induced $k$-th exterior power of $A$, where for every base point $x\in X$, $\wedge^k A (x)$ acts linearly on $\wedge^k \R^d$, the $k$-th exterior power of the Euclidean space $\R^d$. 
	
	Note that the pinching condition on the $k$-th exterior power is equivalent to all the product of $k$ distinct eigenvalues being different, and the twisting condition is equivalent to the action of $g$ on the exterior power having all its algebraic minors different from zero. Both of these conditions are open and dense in the subset of $\alpha$-H\"older fiber bunched cocycles (see~\cite{BoV04}). 
	
	A linear cocycle is called {\em typical} if the pinching and twisting conditions hold for all of its exterior powers.
	
	Note that the maximal Lyapunov exponent of $\wedge^k A$ is the sum of the first $k$ Lyapunov exponents of $A$. 
	We then have the following result.
	
	\begin{theorem}\label{theorem.main.allexp}
		Let $T \colon X\to X$ be the left shift, let $\mu$ be an ergodic measure that has local product structure with H\"older density and let $A \colon X\to \GL$ be an $\alpha$-H\"older continuous fiber bunched linear cocycle. If $A$ is typical then there exists a $C^0$ neighborhood of $A$ such that on this neighborhood, all Lyapunov exponents are H\"older continuous functions of the cocycle.
	\end{theorem}

	\begin{remark}
		Let $(M, f)$ be a hyperbolic homeomorphism. Given a H\"older continuous linear cocycle $A \colon M\to \GL$ we use the H\"older continuous semi-conjugation $\pi$ of $(M, f)$ to a Markov shift $(X, T)$ to obtain the cocycle $A\circ \pi \colon X \to \GL$  over the shift $T$. This allows us to transfer  the results above to cocycles over uniformly hyperbolic homeomorphisms.
	\end{remark}


	\section{Holonomy Reduction}
	\label{holonomy}
	In this section we prove the existence of H\"older holonomies. Moreover we show that up to a conjugacy, the cocycle $A$ can be taken to be constant on unstable sets.
	
	Firstly observe that if $A$ is fiber bunched then there exists $\tau<1$ such that $\norm{A^N(x)}\norm{A^N(x)^{-1}}2^{-N\alpha}<\tau<1$.
	\begin{theorem}\label{t.Holder.holonomies}
		If $A$ is an $\alpha$-H\"older, fiber bunched cocycle, then there exists $0<\beta<\alpha$ such that $A$ admits $\beta$-H\"older stable holonomies. Moreover, $\beta$ depends on $\norm{A},\alpha,\tau,Lip(T)$, so can be taken uniform on a $C^0$ neighborhood of $A$.
	\end{theorem}
	\begin{proof}
		For $x\in \Wsloc (y)$ define $H^n_{x,y}={A^n(y)}^{-1}A^n(x)$ by \cite[Lemme~1.12]{BGV03} when $n$ goes to infinity $H^n_{x,y}$ converges to a linear map $H^s_{x,y}$ that satisfies the first 3 properties on the definition of holonomies, so to prove that is a $\beta$-H\"older holonomy we are left to prove that it also satisfies the $4^{rth}$ property. 
		
		Take $x\in \Wsloc (y)$ and $x'\in \Wsloc (y')$.
		By \cite[Remarque~1.13]{BGV03} there exists $C_1>0$ such that $\norm{H^s_{x,y}-H^n_{x,y}}\leq C_1\tau^n d(x,y)^{\alpha}$,
		so we get that for every $n>0$
		\begin{equation}\label{eq.holder.holonomy}
		\norm{H^s_{x,y}-H^s_{x',y'}}\leq \norm{H^n_{x,y}-H^n_{x',y'}}+2C_1\tau^n.
		\end{equation}
		Now observe that
		$$\norm{H^n_{x,y}-H^n_{x',y'}}\leq\norm{{(A^n)^{-1}}}\norm{A^n(x)-A^n(x')}+\norm{A^n}\norm{A^n(y)^{-1}-A^n(y')^{-1}}, $$
		and let $L=\Lip(T)$, we can estimate 
		$$
		\begin{aligned}
		\norm{A^n(x)-A^n(x')}&\leq \norm{A}^{n-1}\sum_{j=0}^{n-1}\norm{A(T^j(x))-A(T^j(x'))}\\
		&\leq \norm{A}^{n-1}(\sum_{j=0}^{n-1}L^{\alpha j})v_{\alpha}(A)\dist(x,x')^{\alpha}\\
		&\leq	\norm{A}^{n-1}\frac{L^{(n+1)\alpha}}{L^\alpha-1}v_{\alpha}(A)\dist(x,x')^{\alpha}.
		\end{aligned}
		$$
		
		Let $K=\max\{\norm{A},\norm{A^{-1}}\}^2 L^\alpha$ and 
		$$c=2 \max\{v_{\alpha}(A),v_{\alpha}(A^{-1})\}\, \frac{L^{\alpha}}{L^\alpha-1}, $$
		then we get
		$
		\norm{(A^n) ^{-1}}\norm{A^n(x)-A^n(x')}\leq c K^{n} \dist(x,x')^{\alpha}
		$.  We have an analogous inequality for $\norm{A^n(y)^{-1}-A^n(y')^{-1}}$, 
		so from \eqref{eq.holder.holonomy} we get 
		\begin{equation}\label{eq.holder.holonomy2}
		\norm{H^s_{x,y}-H^s_{x',y'}}\leq 2 c K^n \max\{\dist(x,x')^{\alpha},\dist(y,y')^{\alpha}\} +2 C_1 \tau^n.
		\end{equation}
		
		Without loose of generality assume that $\dist(y,y')\leq \dist(x,x')$.
		Take any $\gamma<\alpha$ and $n$ such that 
		$$
		\frac{\gamma-\alpha}{\log K}\log\dist(x,x')-1<n\leq \frac{\gamma-\alpha}{\log K}\log\dist(x,x').
		$$
		
		We get that 
		$$
		\begin{aligned}
		c K^n  \dist(x,x')^{\alpha}+C_1 \tau^n &\leq c \dist(x,x')^\gamma +C_1 \tau^{\frac{\gamma-\alpha}{\log K}\log\dist(x,x')-1} \\
		& = c \dist(x,x')^\gamma +C_1 e^{-1}\dist(x,x')^{\frac{\gamma-\alpha}{\log K}\log \tau},
		\end{aligned}
		$$
		taking $\beta=\min\{\gamma,\frac{\gamma-\alpha}{\log K}\log \tau\}$ we have 
		$$c K^n  \dist(x,x')^{\alpha}+C_1\tau^n<(c+e^{-1}C_1)\dist(x,x')^\beta.$$ 
		So we can conclude that
		\begin{equation}
		\norm{H^s_{x,y}-H^s_{x',y'}}\leq 2 (c+C_1 e^{-1}) \max\{\dist(x,x')^{\beta},\dist(y,y')^{\beta}\}
		\end{equation}
	\end{proof}
	Now we are going to replace $A$ by a cocycle conjugate to $A$ that only depends on the negative part of the sequence, this construction was already made on \cite{AvV1}, we explain it here for completeness.

	\begin{proposition}\label{prop.reduction}
		Given $A:X\to \GL$, $\alpha$-H\"older continuous and fiber bunched, there exists $A^-:X^-\to \GL$ $\beta$-H\"older continuous, where $0<\beta< \alpha$ depends on $\alpha$, $\norm{A}$ and $\tau$, such that $A^s=A^-\circ P^-$ is conjugated to $A\circ T^{-1}$.  Moreover the conjugation is given by $(x,v)\mapsto (x,H(x)v)$, where $H:X\to \GL$ is $\beta$-H\"older continuous, $\norm{H}$ and $\norm{H^{-1}}$ are uniformly bounded on a neighborhood of $A$ and the map $A\mapsto A^s$ is continuous on the uniform topology.
	\end{proposition} 
	\begin{proof}
		For each $i\in \Sigma$ fix $p_i\in [i]$ and define $\theta:X\to X$ as $\theta(x)$ to be the unique point that belongs to $\Wsloc(p_i)\cap \Wuloc(x)$ for $x\in [i]$. Using the unstable holonomies define $A^s:X\to \GL$ as
		$$
		A^s(x)=H^u_{x,\theta(x)} A(T^{-1}x) H^u_{\theta(T^{-1}x),T^{-1}x}=H^u_{T(\theta(T^{-1}x)),\theta(x)} A(T^{-1}\theta(x)),
		$$
		then we have that $A^s(x)$ only depend on the negative part of $x$ so there exists $A^-:X^-\to \GL$ such that $A^s=A^-\circ P^-$. Also $A^s$ is conjugated to $A\circ T^{-1}$ by the function $(x,v)\mapsto (x,H^u_{\theta(T^{-1}x),T^{-1}x}v)$, so $A^-$ has the same exponents as $A$. Also fixing $x,y$ the map $A\mapsto H^u_{x,y}$ is continuous on the uniform topology, then $A\mapsto A^s$ also is.
		
		By theorem~\ref{t.Holder.holonomies} we have that $A^-$ is $\beta$-H\"older. 
	\end{proof} 
	
	From now on we assume that $A$ is of the form $A^-\circ P^-$.

	
	\section{Quasi-compact Markov operators}
	\label{quasicompact}
	 
	In this section we recall concepts like Markov operators, stationary measures
	and provide a sufficient criterion for the strong mixing property based on the quasi-compactness of the Markov operator and the classical theorem of Ionescu-Tulcea and Marinescu.

	\begin{definition}
		A stochastic dynamical system (SDS) on a compact metric space $\Gamma$  is any continuous mapping $K:\Gamma\to \Prob(\Gamma)$, where the space $\Prob(\Gamma)$ of Borel probability measures on $\Gamma$ is endowed with the weak-$\ast$ topology.
	\end{definition}
	
	A SDS $K:\Gamma\to \Prob(\Gamma)$ can be recursively iterated as follows
	$$K^n:\Gamma\to \Prob(\Gamma),\quad  K^n_x :=\int_\Gamma K^{n-1}_y\, dK_x(y) \quad \text{ for }\, n\geq 1,\; x\in\Gamma . $$ 
	\begin{proposition}
		Given a SDS $K:\Gamma\to\Prob(\Gamma)$, the linear operator $\Qop_K:C^0(\Gamma)\to C^0(\Gamma)$ defined by $(\Qop_K \varphi)(x):=\int_\Gamma \varphi\, dK_x$, is  positive, bounded (with norm $\leq 1$)
		and fixes the constant functions.
		Its adjoint $\Qop^\ast$ leaves the convex set $\Prob(\Gamma)$ invariant where 
		$\Qop^\ast \mu= \int K_x\, d\mu(x)$.
	\end{proposition}
	
	\begin{proof}
		Straightforward.
	\end{proof}

	\begin{definition}  $\Qop_K$ is called the Markov operator of $K$.
		We call $K$-stationary measure any measure $\mu\in\Prob(\Gamma)$ such that $\Qop^\ast\mu=\mu$.
	\end{definition}

	The dynamics of the SDS is encapsulated by the operator $\Qop_K$. We have for instance that
	$\Qop_{K^n}=(\Qop_K)^n$ for all $n\in\N$.

	\begin{theorem}[Ionescu-Tulcea, Marinescu]
		\label{ITMT}
		Let $(\Escr_1,\norm{\cdot}_1)$ and
		$(\Escr,\norm{\cdot})$ be complex Banach spaces, $\Escr_1\subseteq \Escr$, and $\Qop:\Escr\to \Escr$ a linear operator such that:
		\begin{enumerate}
			\item $\Qop:\Escr\to \Escr$ is a bounded operator in $(\Escr,\norm{\cdot})$ with $\norm{\Qop}\leq 1$.
			\item If $\varphi_n\in \Escr_1$, $\varphi\in\Escr$, $\lim_{n\to \infty}\norm{\varphi_n-\varphi}= 0$  and $\norm{\varphi_n}_1\leq C$ for all $n\in\N$ then $\varphi\in \Escr_1$ and $\norm{\varphi}_1\leq C$.
			\item The inclusion $i:\Escr_1\to \Escr$ is bounded and compact, i.e., for some constant $M<\infty$, $\norm{\varphi}\leq M\,\norm{\varphi}_1$  for all $\varphi\in\Escr_1$ and any bounded set in $(\Escr_1,\norm{\cdot}_1)$ is relatively compact in $(\Escr,\norm{\cdot})$.
			\item There exist $0<\sigma<1$, $n_0\in\N$ and $C<\infty$ such that for all $\varphi\in \Escr_1$,\; 
			$\norm{\Qop^{n_0} \varphi}_1\leq \sigma^{n_0} \, \norm{\varphi}_1 + C\, \norm{\varphi}$.
		\end{enumerate}
		Then $\Qop$ has a finite set $G$ of eigenvalues with finite multiplicity of absolute value $1$ and the rest of the spectrum of $\Qop$ is a compact subset of the unit disk $\{z\in \C\colon \abs{z}<1\}$.
		Furthermore there exist bounded linear operators 
		$\{L_\lambda:\Escr_1\to \Escr_1\}_{\lambda\in G}$ and $N:\Escr_1\to \Escr_1$ such that 
		$$  \Qop^{n_0} = \sum_{\lambda \in G} \lambda^{n_0} \, L_\lambda + N $$
		and for all $\lambda\in  G$:
		\begin{enumerate}
			\item \; $L_\lambda\circ L_{\lambda'} = 0$, for $\lambda'\in G$ with $\lambda\neq \lambda'$,
			\item \; $L_\lambda\circ L_\lambda=L_\lambda$,
			\item \; $L_\lambda\circ N=N\circ L_\lambda=0$,
			\item \; $L_\lambda(\Escr_1)=\{\varphi\in\Escr_1 : \Qop \varphi=\lambda\,\varphi \}$,
			\item \; $N$ has spectral radius $<1$.
		\end{enumerate}
	\end{theorem}
	
	\begin{proof}
		See~\cite[Theorem 7.1.1]{GB97} or ~\cite{IM50}.

	\end{proof}

	\begin{remark}
		The conclusion of Theorem~\ref{ITMT}  is usually expressed saying that $\Qop$ is a quasi-compact operator.
	\end{remark}

	Let $(\Sigma,d)$ be a compact metric space with diameter $\leq 1$.
	Given $0<\alpha\leq 1$ define the seminorm $v_\alpha:C^0(\Sigma)\to [0,+\infty]$
	$$ v_\alpha(\varphi):= \sup_{x\neq y} \frac{\abs{\varphi(x)-\varphi(y)}}{d(x,y)^\alpha} $$
	and the norm\, 
	$\norm{ \varphi }_\alpha := v_\alpha(\varphi) + \norm{\varphi}_\infty  $.
	The space $(C^\alpha(\Sigma), \norm{\cdot}_\alpha)$
	of  $\alpha$-H\"older continuous observables
	$\varphi:\Sigma\to\R$ is a Banach algebra and vector lattice.
	
	The SDS $K$ determines the point-set map
	\begin{equation}
	\label{def Omega K}
	\Omega_K:\Sigma\to 2^\Sigma,\quad \Omega_K(x):= \overline{\cup_{n\geq 0} \supp(K_x^n)} . 
	\end{equation}
	
	\begin{theorem} 
		\label{base strong mixing}
		Given  $\Sigma$ compact, $\mu\in\Prob(\Sigma)$
		and  $K:\Sigma\to\Prob(\Sigma)$ a SDS with  Markov 
		operator $\Qop$, if 
		\begin{enumerate}
			\item  $\mu$ is a $K$-stationary measure;
			
			\item For every $p\in\N$ and  $x\in \Sigma$, $\supp(\mu)\subseteq \Omega_{K^p}(x) $;

			\item There are constants $0<\alpha \leq 1$,   $n_0\geq 1$, $C<\infty$ and $0<\sigma<1$
			such that for every $\varphi\in C^\alpha(\Sigma)$,
			$$ v_\alpha(\Qop^{n_0}\varphi) \leq \sigma^{n_0} v_\alpha(\varphi)+ C\, \norm{\varphi}_\infty  $$	
		\end{enumerate}
		then $\Qop$ is strongly mixing on the space $C^\alpha(\Sigma)$, i.e.,   there exists  $0<\sigma_0<1$
		such that for all $\varphi\in C^\alpha(\Sigma)$ and some $C_0=C_0(\norm{\varphi}_\alpha)<\infty$,
		$$ \norm{ \Qop^n \varphi-\smallint \varphi\, d\mu }_\infty \leq C_0\, \sigma_0^n  \qquad \forall \, n\in\N .$$
	\end{theorem}

	\begin{proof}
		Consider $(\Escr,\norm{\cdot})=(C^0(\Sigma), \norm{\cdot}_\infty)$ and $(\Escr_1,\norm{\cdot}_1)=(C^\alpha(\Sigma ), \norm{\cdot}_\alpha)$. Then the Markov operator
		$\Qop:C^0(\Sigma)\to C^0(\Sigma)$ satisfies the hypothesis of Theorem~\ref{ITMT} and, therefore, it is a quasi-compact operator.
		
		Notice that assumption (1)  of Theorem~\ref{ITMT}  holds because
		$$ \abs{(\Qop \varphi)(x)}
		\leq \int   \abs{\varphi(y)}\, dK_x(y) \leq \norm{\varphi}_\infty .$$
		If $\varphi_n\in  C^\alpha(\Sigma)$, $\varphi\in C^0(\Sigma)$, $\lim_{n\to \infty}\norm{\varphi_n-\varphi}_\infty= 0$  and $\norm{\varphi_n}_\alpha\leq C$ for all $n\in\N$ then
		taking the limit of the following inequalities
		$$ \frac{\abs{\varphi_n(x)-\varphi_n(y)}}{d(x,y)^\alpha}
		\leq C $$
		we get that $\norm{\varphi}_\alpha\leq C$. In particular $\varphi$ is $\alpha$-H\"older continuous.
		This proves assumption (2) of Theorem~\ref{ITMT}.
		The inclusion $i:C^\alpha(\Sigma)\to C^0(\Sigma)$ is a bounded linear operator because $\norm{\varphi}_\infty\leq \norm{\varphi}_\alpha$. It is a compact operator by Ascoli-Arzel\'a's Theorem. Therefore, assumption (3) of Theorem~\ref{ITMT} holds.
		Finally the assumption (4)  of Theorem~\ref{ITMT} matches hypothesis (3) above.

		\begin{lemma}
			\label{max value lemma} 
			Given $\varphi\in C^0(\Sigma)$, if $\varphi$ is real valued and  $\varphi\leq \Qop^p \varphi$, for some   $p\geq 1$, then $\varphi=c$ is constant   on $\supp(\mu)$  with $\varphi\leq c$ on $\Sigma$.
		\end{lemma}
		
		\begin{proof}
			We assume $p=1$ and  prove that
			if \, $\varphi\leq \Qop \varphi$ \, then \, $\varphi=c$ \, is constant on $\supp(\mu)$ with $\varphi\leq c$ on $\Sigma$. For this we use assumption (2), which for $p=1$ says that $\supp(\mu)\subseteq  \Omega_{K}(x)$ for all $x\in \Sigma$. 
			The general case reduces to this fact 
			applied to the operator
			$(\Qop^p \varphi)(x):=\int \varphi\, dK^p_x$,
			using instead that  by   (2) $ \supp(\mu)\subseteq \Omega_{K^p}(x) $.

			By the Weierstrass principle there are points $x_0 \in \Sigma$ such that $\varphi(x_0)\geq \varphi(x)$ for all $ x\in \Sigma$. Then 
			$$ c=\varphi(x_0)\leq (\Qop \varphi)(x_0)=\int \varphi\, dK_{x_0}   \leq c ,$$
			which implies that $\varphi(y)=c$ for all $y\in \supp(K_{x_0})$.
			Assume next (induction hypothesis) that  $\varphi(y)=c$
			for every $y\in \supp( K_{x_0}^n)$.
			Given $z\in  \supp( K_{x_0}^{n+1})$, we have $z\in \supp(K_{y})$ for some  $y\in  \supp( K_{x_0}^n)$. 
			By induction hypothesis, $\varphi(y)=c$, which by the previous argument implies that $\varphi(z)=c$. Hence, by induction,
			the function $\varphi$ is constant and equal to $c$ on $\Omega_{K}(x_0)$. Since by (1) $\supp(\mu)\subseteq \Omega_{K}(x_0)$, it follows that $\varphi=c$ on  $\supp(\mu)$. On the other hand the inequality
			$\varphi(x)\leq \varphi(x_0)=c$ holds for all $x\in\Sigma$,  by choice of $x_0$.
		\end{proof}

		The following argument uses the notion of peripheral spectrum of  $\Qop$. See H. Schaefer~\cite[Chap. V, $\S$4 and $\S$5]{Sc74}.	
		
		Let $G:= \mathrm{spec}(\Qop) \cap \{\lambda\in \C\colon \abs{\lambda}=1\}$. By Theorem~\ref{ITMT},
		$G$ is a finite set and there exists a 
		direct sum decomposition into closed $\Qop$-invariant  subspaces $C^\alpha(\Sigma)=F\oplus H$ such that
		$\dim (F)<\infty$,  $\mathrm{spec}(\Qop\vert_F)=G$ and 
		$\Qop\vert_H$ has spectral radius $<1$.

		For each $\lambda\in G$ choose a (complex) eigenfunction $f_\lambda\in C^\alpha(\Sigma)$ such that
		$\norm{f_\lambda}_\infty=1$ and $\Qop f_\lambda=\lambda\, f_\lambda$. Since $\abs{\lambda}=1$, for all $x\in \Sigma$,
		$$ \abs{f_\lambda(x)} = \abs{\lambda\, f_\lambda(x)}
		=\abs{(\Qop f_\lambda)(x)}\leq (\Qop \abs{f_\lambda})(x) ,$$
		which by Lemma~\ref{max value lemma}  implies that $\abs{f_\lambda}$ is constant  and equal to $c=\norm{f_\lambda}_\infty=1$ on $\supp(\mu)$.
		In particular, because $f_\lambda$ takes values in the unit circle  a convexity argument implies   $\lambda\,f_\lambda(x)=\Qop(f_\lambda)(x) = f_\lambda(y)$, for all $x \in \Sigma$ and $y\in \supp(K_x)$.
		Thus, given $\lambda, \lambda'\in G$ and  $x\in \Sigma$,
		\begin{align*}
		\Qop (f_\lambda\, f_{\lambda'})(x) &=   \int f_\lambda(y)\, f_{\lambda'} (y)\, dK_x(y) \\
		&=(\lambda\, f_\lambda(x))\,(\lambda' \, f_{\lambda'}(x))
		=(\lambda\, \lambda')\, (f_\lambda\, f_{\lambda'})(x) .
		\end{align*}
		This implies that $G$ is a group. In particular
		every eigenvalue $\lambda\in G$ is a root of unity.
		Assume now that there exists an eigenvalue $\lambda\in G$, $\lambda\neq 1$,
		with $\lambda^n=1$ and consider the corresponding eigenfunction  $f_\lambda$. Since $\lambda\neq 1$,  $f_\lambda$ can not be constant. On the other hand, because
		$\Qop f_\lambda=\lambda\, f_\lambda$, 
		$\Qop^nf_\lambda=\lambda^n\, f_\lambda=f_\lambda$,
		and applying Lemma~\ref{max value lemma} to the real and imaginary parts of $f_\lambda$ we conclude that $f_\lambda$ must be constant. Obviously $f_\lambda$ is constant on $\supp(\mu)$, but more can be said. First, arguing as above $f_\lambda(x)=f_\lambda(y)$ for all $y\in \supp(K_x)$.
		By induction,  $f_\lambda(x)=f_\lambda(y)$, for all $y\in \supp(K_x^n)$ and $n\in\N$. Finally, by assumption (2),
		$f_\lambda(x)=f_\lambda(y)$, for all $x\in\Sigma$ and $y\in \supp(\mu)$. This proves that $f_\lambda$ is globally constant.	
		This contradiction proves that $G=\{1\}$.
		Finally, Lemma~\ref{max value lemma} again implies that $\lambda=1$ is a simple eigenvalue.
		
		The previous considerations imply that
		$F=\{\text{constant functions}\}$ and
		$H= \{\varphi\in C^\alpha(\Sigma)\colon \smallint \varphi\, d\mu=0\}$, where the spectral radius of $\Qop\vert_H$ is $<1$. Hence there exist constants $C_0<\infty$ and  $0<\sigma_0<1$ such that  $\norm{\Qop^n \varphi}_\alpha \leq C_0\, \sigma_0^n \, \norm{\varphi}_\alpha$
		for all $n\in\N$ and $\varphi\in H$.
		
		Finally, given $\varphi\in C^\alpha(\Sigma)$, 
		\begin{align*}
		\norm{\Qop^n\varphi -\smallint \varphi\, d\mu}_\infty &\leq  \norm{\Qop^n\varphi -\smallint \varphi\, d\mu}_\alpha  
		= \norm{\Qop^n( \varphi -\smallint \varphi\, d\mu )}_\alpha\\ 
		&\leq C_0\, \sigma_0^n\, \norm{\varphi-\smallint \varphi\, d\mu}_\alpha \leq C_0\left(\norm{\varphi}_\alpha +\smallint \vert\varphi\vert  \, d\mu\right)\, \sigma_0^n .
		\end{align*}
	\end{proof}


	\section{Base Strong Mixing}
	\label{basemixing}
	In this section we establish the strong mixing of the Markov operator associated to the base dynamics, and as a consequence, we prove a base large deviations theorem.

	Let  $X^-:=\Sigma^{-\N\cup \{0\}}$.
	Given $x^-\in X^-$, we denote by $(x^-,i)$   the sequence obtained shifting $x^-$  one place to the left and inserting the symbol $i$ at position $0$. This operation is an inverse branch of the right shift   $T_-$ on $X^-$.

	We make the interpretation $X^-\equiv X/{\Wuloc}$, so that each point $x^-\in X^-$ represents the local unstable manifold
	$$ W^u_{x^-}:=\{ x\in X :  x^-_j=x_j,\, \forall j\leq 0  \}  $$
	of any point $x\in W^u_{x^-}$, i.e., $W^u_{x^-}=\Wuloc(x)$ $\forall \, x\in W^u_{x^-}$.  From now on we will write  $W^u_{loc}(x^-)$ instead of $W^u_{x^-}$.

	The $(d-1)$-dimensional simplex in $\R^d$ is denoted by
	$$\Delta^{d-1}:=\{(p_1,\ldots, p_d)\colon p_j\geq 0, \, \sum_{j=1}^dp_j=1 \,\} . $$
	
	The measure $\mu$ determines  
	$p:X^-\to \Delta^{d-1}$ with components $p=(p_1,\ldots, p_d)$
	defined by
	$$p_i(x^-):= \muxminus (W^u_{x^-}\cap T^{-1} W^u_{(x^-,i)})  $$ 
	where $\{\muxminus\}_{x^-\in X^-}$ stands for the disintegration of $\mu$
	over the partition $\Pscr^u:=\{ W^u_{x^-} : x^-\in X^-\}$.

	\begin{definition}
		\label{admissible sequences}
		The set of admissible sequences in $X$ is
		$$ \mathscr{A} :=\cap_{n\in\Z} T^{-n}\{x\in X\colon p_{x_1}(x^-)>0 \} .$$
		Similarly, the set of admissible sequences in $X^-$ is
		$$ \mathscr{A}^- :=\cap_{n\geq 0} (T_-)^{-n}\{x\in X\colon p_{x_0}(T_-(x))>0 \} $$
		where $T_-:X^-\to X^-$ denotes the right shift.
	\end{definition}
	
	\begin{remark}
		$P^-(\mathscr{A})\subseteq \mathscr{A}^-$. 
	\end{remark}

	\begin{proposition}
		\label{prop admissible}
		The sets of admissible sequences $\mathscr{A}$ and $\mathscr{A}^-$ have
		full measure w.r.t. $\mu$ and $\mu^-:=(P^-)_\ast\mu$, respectively.
	\end{proposition}
	
	\begin{proof}
		Straightforward.
	\end{proof}

	\begin{proposition}
		\label{p alpha Holder}
		The assumptions on $\mu$ imply that 
		$p:X^-\to \Delta^{d-1}$ is $\alpha$-H\"older continuous function for some $0<\alpha\leq 1$.
	\end{proposition}
	
	\begin{proof}
		With the notation of Definition~\ref{def local product measure}
		we claim that on the cylinder $[i]:=\{ y \in X  \colon y_0=i\}$, 
		$\muxminus=\rho(x^-,\cdot)\, \mu_i^+$.
		For this we need to check that
		the family of measures $\muxminus:=\sum_{i=1}^\ell\rho(x^-,\cdot)\, \mu_i^+$  satisfies for any bounded  observable $\varphi:X\to \R$ 
		\begin{align*}
		\int \varphi\, d\mu &= \sum_{i=1}^\ell \int \varphi\, d\mu_{[i]}    =\sum_{i=1}^\ell \int \varphi(x^-, x^+)\, d(h_\ast\mu_{[i]})(x^-, x^+) \\
		&= \sum_{i=1}^\ell \int \int \varphi(x^-, x^+)\,\rho(x^-,x^+) d\mu^-_i(x^-)\, d\mu^+_i(x^+)\\
		&= \sum_{i=1}^\ell \int \int \varphi(x^-, x^+)\, d(\rho(x^-,\cdot ) \,\mu^+_i)(x^+)\, d\mu^-_i(x^-)\\
		&= \sum_{i=1}^\ell \int \int \varphi(x^-, x^+)\, d\muxminus(x^+)\, d\mu^-_i(x^-)\\
		&= \sum_{i=1}^\ell \int\left( \int \varphi\, d\muxminus\right) \, d\mu^-_i(x^-)\\
		&=  \int\left( \int \varphi\, d\muxminus \right) \, d\mu^- (x^-) 
		\end{align*}
		the last equality  because  $\mu^-=
		(P_-)_\ast \left(\sum_{i=1}^\ell \mu_{[i]}  \right) = \sum_{i=1}^\ell \mu^-_i$. Thus since
		\begin{align*} p_i(x^-) &=\muxminus(W^u_{x^-}\cap T^{-1} W^u_{(x^-,i)}) =\muxminus( [i])   \\
		&= \int \rho(x^-, x^+)\, d\mu^+_i(x^+)
		\end{align*}
		and $\rho$ is H\"older continuous, it follows  that $p_i:X^-\to [0,+\infty) $ is  H\"older continuous for all $i\in\{1,\ldots, \ell\}$.
	\end{proof}

	The space of sequences $X^-$ is a compact metric space when endowed with the distance $d:X^-\times X^-\to \R$, $d(x^-,y^-):= 2^{\kappa(x^-, y^-)}$ where
	$$ \kappa(x^-,y^-):= \min \{ n\in\N \colon \, x^-_{-n}\neq y^-_{-n}\, \} , $$
	with the convention that $\min\emptyset=-\infty$.
	We define the SDS
	\begin{equation}
	\label{base SDS}
	K:X^-\to \Prob(X^-)\quad \text{ by }\quad K_{x^-}:= \sum_{j=1}^\ell p_i(x^-)\, \delta_{(x^-,i)} 
	\end{equation}
	which in turn determines the Markov  operator
	$\Qop:C^0(X^-)\to C^0(X^-)$,
	\begin{equation}
	\label{base Qop}
	(\Qop \varphi)(x^-)= \sum_{i=1}^\ell p_i(x^-)\, \varphi((x^-,i)).
	\end{equation}

	\begin{proposition}
		\label{mus, stationary measure}
		$\Qop^\ast \mu^-=\mu^-$ is a stationary measure. 
	\end{proposition}

	\begin{proof}
		Using the notation in the proof of Proposition~\ref{p alpha Holder}, for any $\varphi\in C^0(X^-)$,
		\begin{align*}
		\int \varphi\, d\Qop^\ast \mu^-  &= \int (\Qop \varphi) \, d\mu^-  = \sum_{i=1}^\ell \int p_i(x^-)\, \varphi((x^-, i))\, d\mu^-(x^-) \\
		&= \sum_{i=1}^\ell  \int  \varphi((x^-, i))\, \muxminus([i])\, d\mu^-(x^-) \\
		&= \sum_{i=1}^\ell  \int \left(  \int \varphi((x^-, i))\, \ind_{[i]}(y^+) \, d\muxminus(y^+)  \right)\,  d\mu^-(x^-) \\
		&=  \iiint   \varphi((x^-, y_1))\, d\mu(x^-,y_1, y^+) \\
		&= \int \varphi\circ P_-\, d\mu = \int \varphi\, d\mu^- ,
		\end{align*}
		which implies that $\Qop^\ast\mu^-= \mu^-$.
	\end{proof}

	We introduce a couple of seminorms and norms on  $C^0(X^-)$. 
	For each $k\in\N$ define
	$v_k:C^0(X^-)\to [0,\infty]$,
	$$ v_k(\varphi):= \sup\{ \abs{\varphi(x^-)-\varphi(y^-)} \colon x^-,y^-\in X^- \, \text{ s.t. }\, x^-_{j}=y^-_{j},\, \forall\, j \ge -k \, \}. $$
	
	Given $0<\alpha\leq 1$, define $v_\alpha:C^0(X^-)\to [0,\infty]$,
	$$ v_\alpha(\varphi):= \sup\{2^{k\,\alpha}\,v_k(\varphi)  \colon k\in\Z,\, k\geq 0 \, \}. $$
	This function is a seminorm, also characterized by
	$$ v_\alpha(\varphi)=\sup\left\{ \frac{\abs{\varphi(x^-)-\varphi(y^-)}}{d(x^-,y^-)^\alpha} \colon x^-, y^-\in X^-, \, x^-\neq y^- \right\} . $$
	
	Denote by $C^\alpha(X^-)$ the Banach space of $\alpha$-H\"older continuous functions $\varphi:X^-\to \C$
	endowed with the norm
	$$ \norm{\varphi}_\alpha:= v_\alpha(\varphi)+\norm{\varphi}_\infty . $$
	The space $C^\alpha(X^-)$ with this norm is  a Banach algebra, which means that $\norm{1}_\alpha=1$ and 
	$\norm{\varphi\,\psi}_\alpha\leq \norm{\varphi}_\alpha\, \norm{\psi}_\alpha$.

	\begin{proposition}
		\label{v alpha Qop varphi ineq }
		Choose $0<\alpha\leq 1$ according to Proposition~\ref{p alpha Holder}. Then for any $\varphi\in C^\alpha(X^-)$,
		$$ v_\alpha(\Qop^n \varphi) \leq 2^{-n\,\alpha}\, v_\alpha(\varphi) + \frac{v_\alpha(p)}{1-2^{-\alpha}}\, \norm{\varphi}_\infty .$$
	\end{proposition}

	\begin{proof}
		Assume $x^-, y^-\in X^-$ with $x^-_{j}=y^-_{j},\, \forall\, j \geq -k $. Then for any $i\in \{1,\ldots, \ell\}$ the first coordinates of the sequences
		$(x^-,i)$ and $(y^-,i)$ match up to order $k+1$. Hence
		\begin{align*}
		\abs{(\Qop \varphi)(x^-) - (\Qop \varphi)(y^-) } 
		&\leq  \sum_{i=1}^\ell \abs{ p_i(x^-)\, \varphi((x^-,i)) -   p_i(y^-)\, \varphi((y^-,i)) }\\
		&\leq  \sum_{i=1}^\ell p_i(x^-)\, \abs{ \varphi((x^-,i)) -    \varphi((y^-,i)) } \\
		&\qquad +  \sum_{i=1}^\ell \abs{p_i(x^-)-p_i(y^-)} \,  \abs{\varphi((y^-,i)) }\\
		&\leq v_{k+1}(\varphi) + v_k(p)\, \norm{\varphi}_\infty .
		\end{align*}
		Thus, taking the supremum in $x^-$ and $y^-$ as above we get
		$$  v_k(\Qop \varphi) \leq v_{k+1}(\varphi) + v_k(p)\, \norm{\varphi}_\infty .$$
		Given $k\in \N$ we now have
		\begin{align*}
		2^{k\,\alpha}\, v_k(\Qop \varphi) &\leq 
		2^{k\,\alpha}\, v_{k+1}(\varphi) + 2^{k\,\alpha}\, v_k(p)\, \norm{\varphi}_\infty\\
		&\leq 
		2^{-\alpha}\, v_{\alpha}(\varphi) +  v_\alpha(p)\, \norm{\varphi}_\infty .
		\end{align*}
		Hence taking the supremum in $k$
		$$ v_\alpha(\Qop \varphi)\leq 
		2^{-\alpha}\, v_{\alpha}(\varphi) +  v_\alpha(p)\, \norm{\varphi}_\infty . $$
		Finally by induction we get
		\begin{align*}
		v_\alpha(\Qop^n \varphi) &\leq 
		2^{-n\,\alpha}\, v_{\alpha}(\varphi) + (1+2^{-\alpha}+\cdots +2^{-(n-1)\alpha} )\, v_\alpha(p)\, \norm{\varphi}_\infty\\
		&\leq 
		2^{-n\alpha}\, v_{\alpha}(\varphi) +  \frac{v_\alpha(p)}{1-2^{-\alpha}}\, \norm{\varphi}_\infty .
		\end{align*}
	\end{proof}

	The operator $\Qop$ is strongly mixing on $C^\alpha(X^-)$.

	\begin{proposition}
		\label{Qop quasi-compact}
		Given  $0<\alpha\leq 1$ according to Proposition~\ref{p alpha Holder}, the Markov operator $\Qop$ is strongly mixing on the space $C^\alpha(X^-)$, i.e.,   there exists   $0<\sigma_0<1$
		such that for any $\varphi\in C^\alpha(X^-)$ and some $C_0=C_0(\norm{\varphi}_\alpha)<\infty$, 
		$$ \norm{ \Qop^n \varphi-\smallint \varphi\, d\mu^- }_\infty \leq C_0\, \sigma_0^n  \qquad \forall \, n\in\N .$$
	\end{proposition}

	\begin{proof}
		Consider the Banach spaces
		$$(\Escr,\norm{\cdot})=(C(X^-), \norm{\cdot}_\infty)\;  \text{and } \; (\Escr_1,\norm{\cdot}_1)=(C^\alpha(X^-), \norm{\cdot}_\alpha) .$$ 
		Since the Markov operator
		$\Qop:C(X^-)\to C(X^-)$ satisfies the assumptions of Theorem~\ref{ITMT} it is a quasi-compact operator.
		The strong mixing property will follow from Theorem~\ref{base strong mixing} and we are reduced to check the hypothesis of this theorem. Notice that Proposition~\ref{mus, stationary measure} implies hypothesis (1), while Proposition~\ref{v alpha Qop varphi ineq } implies hypothesis (3) of Theorem~\ref{base strong mixing}.
		We are left to check hypothesis (2).
		For any $x^-\in X^-$ the  set $\Omega_K(x^-)$ contains
		$$\mathscr{A}(x^-):= \{ (x^-, i_1, \ldots , i_{n} ) \colon   n\in\N, \; \forall \,  k\leq n \quad  p_{i_k}(x^-, i_1, \ldots , i_{k-1} ) >0   \}  $$
		which by the ergodicity of $(T_-,\mu^-)$ is dense in the full measure set of all admissible sequences (see Proposition~\ref{prop admissible}). Therefore, $\supp(\mu^-)\subseteq  \Omega_K(x^-)$. To check assumption (2) of Theorem~\ref{base strong mixing} we still need to prove that
		$\supp(\mu^-)\subseteq  \Omega_{K^p}(x^-)$ for every $p\in\N$.
		This follows because $(T_-,\mu^-)$ is mixing and $(T_-^p,\mu^-)$ is ergodic. The expanding map $T_-^p:X^-\to X^-$ can be regarded as a one-step right shift on the space of sequences in the finite alphabet
		$\{1,\ldots, \ell\}^p$, while $\mu^-$ is a $T^p$-invariant probability  measure on $X^-$ with local product structure.
		Therefore  the same argument above proves that $\supp(\mu^-)\subseteq \Omega_{K^p}(x^-)$.
	\end{proof}

	\begin{theorem}
		\label{base ldt}
		Take the H\"older exponent  $\alpha>0$ according to Proposition~\ref{p alpha Holder}. 
		Then there exist constants $C<\infty$ and $k>0$   such that
		for all\,  $\psi \in C^\alpha(X)$ with  $\norm{\psi}_\alpha\leq L$, $0<\varepsilon<1$  and  $n\in\N$,  
		$$  \mu \left\{ x \in X  \colon \, \abs{\frac{1}{n}\,\sum_{j=0}^{n-1} \psi(T^j x) - \int \psi\, d\mu } > \varepsilon \,\right\} \leq C\,e^{ - k\, L^{-2}\, \varepsilon^{2}\, n  } \;. $$
	\end{theorem}

	\begin{proof}
		Let $K\colon X^- \to \Prob(X^-)$ be the SDS defined in~\eqref{base SDS} and consider its stationary measure
		$\mu^-:=(P^-)_\ast\mu$.
		Let $\Xspace$ be the set of all triples $(K,\mu^-, \varphi)$
		with  $\varphi\in C^\alpha(X^-)$. Defining the distance
		$d\left((K,\mu^-, \varphi), (K,\mu^-, \psi) \right) :=\norm{\varphi-\psi}_\alpha$, the set $\Xspace$ becomes a metric space.
		Consider the family of Banach algebras
		$( C^\alpha(X^-), \norm{\cdot}_\alpha))$, indexed in $\alpha\in [0,1]$. This scale satisfies assumptions (B1)-(B7) in~\cite[Section 5.2.1]{DK-book}. 
		The metric space of observed Markov systems  $\Xspace$   satisfies assumptions   (A1)-(A4) in~\cite[Section 5.2.1]{DK-book}. Assumption (A1) holds trivially. Assumption (A2) follows from Proposition~\ref{Qop quasi-compact}.
		Assumptions (A3) and (A4)  hold  easily (see the proof of~\cite[Proposition 5.15]{DK-book}).

		Given any $\varphi\in C^\alpha(X^-)$,
		by~\cite[Theorem 5.4]{DK-book}) there exists a neighbourhood $\mathscr{V}=B_\delta(\varphi)$ and there are positive constants
		$C <\infty$, $k$ and $\varepsilon_0$, depending only on $\norm{\varphi}_\alpha$, such that for all\, $0<\varepsilon<\varepsilon_0$,  $\psi \in\mathscr{V}$ and  $n\in\N$,  
		\begin{equation}
		\label{Markov process base ldt}
		\Pp \left[ \; \abs{\frac{1}{n}\,S_n(\psi) - \E_{\mu^-}(\psi)  } > \varepsilon \,\right]  \leq C\,e^{ - {k\,\varepsilon^2 }\, n  } \;, 
		\end{equation}  
		where $\Pp\in\Prob( (X^-)^\N )$ is any probability measure which makes the process $\{e_n:(X^-)^\N\to X^-\}_{n\geq 0}$,
		$e_n(x_j^-)_{j\geq 0}:= x_n^-$, a stationary Markov process with transition stochastic kernel $K$ and constant common distribution $\mu^-$. The constraint   $\varepsilon<\varepsilon_0$  can be removed replacing the constant $k$ by 
		$k'=k\, \varepsilon_0^2$ so that~\eqref{Markov process base ldt}
		holds for all $0<\varepsilon<1$. The constants $C$, $k$ and  $\delta>0$ (size  of the neighbourhood $\mathscr{V}$) depend basically on the mixing rate in Proposition~\ref{Qop quasi-compact} and the norm 
		$\norm{\varphi}_\alpha$. Hence~\eqref{Markov process base ldt} holds with the same constants  for all $\psi\in C^\alpha(X^-)$ with  $\norm{\psi}_\alpha\leq 1$. 
		
		More generally, given  $\psi\in C^\alpha(X^-)$ with  $\norm{\psi}_\alpha\leq L$, and applying ~\eqref{Markov process base ldt} 
		with $\bar \psi=L^{-1}\psi$ and $\bar \varepsilon=L^{-1}\varepsilon$,
		we get that ~\eqref{Markov process base ldt} still holds with $\bar \varepsilon$ in place of $\varepsilon$. This proves the the large deviation estimates in the Theorem' statement but w.r.t. a probability  $\Pp$ on $(X^-)^\N$ as above.

		Next consider the map $\pi:X\to (X^-)^\N$, $\pi(x_j)_{j\in \Z}:=\{x_n^-\}_{n\geq 0}$ where $x_n^-=( x_{n+j})_{j\leq 0} $. This projection makes  the following diagram commutative
		$$ \begin{CD} X    @>T>>  X\\@V\pi VV        @VV\pi V\\ (X^-)^\N     @>>T >   (X^-)^\N\end{CD} $$
		where the horizontal arrows stand for the left shift maps.
		Defining $\mathrm{e}\colon (X^-)^\N \to X^-$,
		$\mathrm{e}\{x_n^-\}_{n\geq 0}:= x_0^-$, the above process $\{e_n\}_{n\geq 0}$ is given  by
		$e_n= \mathrm{e}\circ T^n$. The process
		$\{e_n\}_{n\geq 0}$   becomes  Markov with common distribution $\mu^-$  for $\Pp :=\pi_\ast\mu$. In particular the previous large deviation estimates hold
		w.r.t. this probability, for all observables in $C^\alpha(X^-)$.
		Given $\psi \in C^\alpha(X^-)$ we make the identification $\psi\equiv \psi\circ P_-$, thus regarding $\psi$ as a function on $X$.
		Since $\pi:X\to (X^-)^\N$ preserves measure,
		for all  $\varepsilon>0$ and $n\in\N$, if $\norm{\psi}_\alpha\leq L$,
		\begin{equation}
		\label{Markov process base ldt 2}
		\mu \left\{  x\in X\colon \; \abs{\frac{1}{n}\, \sum_{j=0}^{n-1} \psi(T^j x) - \int \psi\,d\mu   } > \varepsilon \,\right\}  \leq C\,e^{ - {k\,L^{-2}\, \varepsilon^2 }\, n  } \; . 
		\end{equation} 
		We have established the theorem for observables $\psi\in C^\alpha(X^-)$, i.e., observables which do not depend on future coordinates.

		For the general case the idea is that any observable
		$\psi\in C^\alpha(X)$ is co-homologous to another H\"older observable
		$\psi^- \in C^\beta (X^-)$ for some $0 < \beta < \alpha$, and the large deviation estimates can be transferred  over  co-homologous observables.

		Observales $\psi\in C^\alpha(X)$ can be regarded as (additive) $1$-dimensional  linear cocycles. In particular we can associate them stable and unstable holonomies. Next we outline the reduction procedure of	Section~\ref{holonomy} applied to observables.
		Given $x,y\in X$ with $y\in W^u_{loc}(x)$, we define
		$$h^u_\psi(x,y):= \sum_{n=1}^\infty \psi(T^{-n} y) - \psi(T^{-n} x) . $$
		Since $\psi$ is H\"older continuous and $y\in W^u_{loc}(x)$ this series converges geometrically. Given $x,y,z$ in the same local unstable manifold, the following properties hold, the last one if $T y \in W^u_{loc}(T x)$:
		\begin{itemize}
			\item[(a)]  $ h^u_\psi(x,x) = 0$,
			\item[(b)]  $ h^u_\psi(x,y) = - h^u_\psi(y,x) $,
			\item[(c)]  $ h^u_\psi(x,z) = h^u_\psi(x,y)  +  h^u_\psi(y,z) $,
			\item[(d)]  $ h^u_\psi(x,y) + \psi(y) = \psi(x) + h^u_\psi(T x, T y)$.
		\end{itemize}
		Using these properties and the map $\theta:X\to X$ introduced in Section~\ref{holonomy} we define $\psi^-:X\to \R$ by
		$$ \psi^- (x):= h^u_\psi(\theta(T^{-1} x), T^{-1} x) + \psi(T^{-1} x) + h^u_\psi(x,\theta(x)) . $$
		
		By Theorem~\ref{t.Holder.holonomies}, the map $\theta:X\to X$ is also H\"older,
		which implies that $x\mapsto h^u_\psi(x,\theta(x))$ is $\beta$-H\"oder for some $0<\beta<\alpha$. Therefore 
		$\psi^-\in C^\beta(X)$.
		For the sake of simplicity we assume that with the same parameters $C$ and $k$, the bound~\eqref{Markov process base ldt 2} holds for observables in $C^\beta(X^-)$.
		
		Using the holonomy properties above,
		\begin{align*}
		\psi^{-}(x) &= h^u_\psi(\theta(T^{-1} x), T^{-1} x) + \psi(T^{-1} x) + h^u_\psi(x,\theta(x))\\
		&= h^u_\psi(\theta(T^{-1} x), T^{-1} x) +  h^u_\psi(T^{-1} x,T^{-1}\theta(x)) + \psi(T^{-1} \theta(x) ) \\
		&= h^u_\psi(\theta(T^{-1} x) ,T^{-1}\theta(x)) + \psi(T^{-1} \theta(x) )
		\end{align*}
		which shows that $\psi^-$ does not depend on future coordinates.
		
		Finally we have
		\begin{align*}
		\psi^-(x) - \psi(T^{-1} x) &= h^u_\psi(\theta(T^{-1} x), T^{-1} x) + h^u_\psi(x,\theta(x))\\
		&= h^u_\psi(x,\theta(x)) - h^u_\psi(T^{-1} x, \theta(T^{-1} x) )\\
		&= \eta(x) - \eta(T^{-1} x)
		\end{align*}
		with $\eta(x):= h^u_\psi(x,\theta(x))$. This allows us to transfer over the large deviation estimates from $\psi^-$ to $\psi$.
	\end{proof}

	\section{Fiber Strong Mixing}
	\label{fibermixing}
	In this section we establish the strong mixing of the Markov operator associated to the projective cocycle.

	Consider the space of sequences  $X^-$ introduced at the beginning of  Section~\ref{statements}
	and let $A \colon X^-\to \GL$ be an $\alpha$-H\"older continuous function. 
	
	This function determines an invertible
	cocycle $F\colon X\times \R^d\to X\times\R^d$,
	$F(x, v):=(T x, A(x)\,v)$, where  
	$A(x)$ is short notation for $A(P_-(x))=A(x^-)$.

	Assume that this cocycle satisfies the twisting and pinching condition (see Definition~\ref{pinching and twisting}), which in particular, by~\cite{BoV04}  will imply that there is a gap between the   Lyapunov exponents 
	$L_1(A,\mu)>L_2(A,\mu)$.

	\bigskip

	We define the SDS,  
	$K:X^-\times\proj  \to \Prob(X^- \times\proj)$,
	\begin{equation}
	\label{def Fiber kernel}
	K_{(x^-,\hat p)}:= \sum_{i=1}^\ell p_i(x^-)\, \delta_{((x^-,i), \hat A(x^-)\, \hat p)}. 
	\end{equation}
	This determines the operator
	$\Qop:C^0(X^-\times\proj)\to C^0(X^-\times\proj)$,
	$$ (\Qop \varphi)(x^-,\hat p )= \sum_{i=1}^\ell  p_i(x^-)\, \varphi((x^-,i), \hat A(x^-)\,\hat p). $$
	Here $\hat A (x^-)$ stands for  the projective map induced
	by  $A(x^-)\in \GL$.

	By the pinching and twisting assumption combined with  the Oseledets theorem
	there exists an $F$-invariant  measurable decomposition  $\R^d= E^u(x)\oplus E^{s}(x)$,   where $E^u(x)$ is the Oseledets subspace corresponding to the largest Lyapunov exponent and $E^s(x)$ is the sum of the subspaces corresponding to lower Lyapunov exponents.
	More precisely we set
	
	\begin{align}
	\label{Eu def}
	E^u(x) & :=\left\{ v \in \R^d \colon \lim_{n\to +\infty} \frac{1}{n}\,\log \norm{A^{-n}(x)\, v}= -L_1(A,\mu)\right\} ,\\
	\label{Es def}
	E^s(x) & :=\left\{ v \in \R^d \colon \lim_{n\to +\infty} \frac{1}{n}\,\log \norm{A^n(x)\, v}<L_1(A,\mu)\right\} . 
	\end{align}
	
	We say that a point $x\in X$ is a \textit{$u$-regular point}, in the sense of Oseledet, if for any $v\in\R^d$ the limit in~\eqref{Eu def} exists and $E^u(x)$ is a $1$-dimensional subspace. We say that  $x\in X$ is a \textit{$s$-regular point} if for any $v\in\R^d$ the limit in~\eqref{Es  def} exists and $E^s(x)$ is a codimension $1$ subspace. Finally we say that $x$ is  a \textit{ regular point }  if it is both $u$-regular, $s$-regular 
	and $\R^d=E^u(x)\oplus E^s(x)$. 
	
	Denote by $\mathscr{O}^u$, $\mathscr{O}^s$ and $\mathscr{O}$ the (full measure) sets of all $u$-regular, $s$-regular and  regular points, respectively.
	
	\pagebreak
	
	\begin{proposition}
		\label{constant oseledets}
		Given $y\in \Wuloc(x)$
		\begin{enumerate}
			\item[(1)] $x\in \mathscr{O}^u$ $\Leftrightarrow$ $y\in \mathscr{O}^u$;
			\item[(2)] $x\in \mathscr{O}^u$ or $y\in \mathscr{O}^u$  $\Rightarrow$  $E^u(x)=E^y(y)$
		\end{enumerate}
		Similarly, given $y\in \Wsloc(x)$
		\begin{enumerate}
			\item[(3)] $x\in \mathscr{O}^s$ $\Leftrightarrow$ $y\in \mathscr{O}^s$;
			\item[(4)] $x\in \mathscr{O}^s$ or $y\in \mathscr{O}^s$  $\Rightarrow$  $H^s_{x,y} E^s(x)=E^y(y)$.
		\end{enumerate}
	\end{proposition}

	\begin{proof}
		Since $A$ factors through $P_-$, any negative power of $A$
		$$A^{-n}(x)= A(T^{-n} x)^{-1} \, \cdots \, A(T^{-2} x)^{-1}\, A(T^{-1} x)^{-1}   $$
		is constant along $\Wuloc(x)$. Hence
		$A^{-n}(y)=A^{-n}(x)$ for all $n\in\N$, and items (1) and (2) follow.
		
		Items (3) and (4) follow from the holonomy relation
		$$ A^n(y)= H^s_{T^n x, T^n y} \, A^n(x)\, H^s_{y,x} \quad \forall\,  n\geq 0\quad \forall\,  y\in \Wsloc(x) $$
		and the fact that the holonomies $H^s_{x,y}$ are uniformly bounded.
	\end{proof}

	Because of the previous proposition we can write $E^u(x^-)$ instead of $E^u(x)$.
	Let us denote by $\m\in \Prob(X^-\times\proj)$
	the  measure which admits the disintegration $\{ \m_{x^-}:=\delta_{E^u(x^-)}\colon  x^-\in X^-\}$, over the canonical projection $\pi\colon  X^-\times \proj\to X^-$.

	\begin{proposition}
		\label{stationary measure}
		The operator $\Qop$  admits the stationary measure $\m$. 
	\end{proposition}

	\begin{proof}
		Let $\Qop^\ast$ be the adjoint of $\Qop$ on $C^0( X^-\times\proj)$.
		Denote by $\eu(x^-)$ the projective point that represents the Oseledet's direction $E^u(x^-)$.
		Using the notation of the proof of Proposition~\ref{p alpha Holder}, for any $\varphi\in C^0( X^-\times\proj)$ we have
		
		\begin{align*}
		\int \varphi\, d\Qop^\ast \m  &= \int (\Qop \varphi) \, d\m  \\
		&= \int  \sum_{i=1}^\ell p_i(x^-)\, \varphi((x^-, i), \hat A(x^-) \hat p )\, d\m(x^-,\hat p ) \\
		&= \sum_{i=1}^\ell \int p_i(x^-)\, \varphi((x^-, i), \hat A(x^-) \eu(x^-) )\, d\mu^-(x^- ) \\
		&= \sum_{i=1}^\ell \int p_i(x^-)\, \varphi((x^-, i),   \eu(x^-,i) )\, d\mu^-(x^- ) \\
		&= \sum_{i=1}^\ell  \int  \varphi((x^-, i),   \eu(x^-,i))\, \muxminus([i])\, d\mu^-(x^-) \\
		&= \sum_{i=1}^\ell  \int \left(  \int \varphi((x^-, i),   \eu(x^-,i))\, \ind_{[i]}(y^+) \, d\muxminus(y^+)  \right)\,  d\mu^-(x^-) \\
		&=  \iiint   \varphi((x^-, y_1),   \eu(x^-,y_1))\, d\mu(x^-,y_1, y^+) \\
		&=  \int   \varphi(P_-(x),   \eu(P_-(x)))\, d\mu(x) \\
		&=  \int   \varphi(x^-,   \eu(x^-))\, d\mu^-(x^-)  = \int \varphi\, d\m .
		\end{align*}
		Since $\varphi$ is arbitrary, this implies that $\Qop^\ast\m= \m$.
	\end{proof}

	The projective distance $\delta:\proj\times\proj\to [0,1]$ is defined by
	$$ \delta(\hat p, \hat q):= \norm{p\wedge q}=\abs{\sin \angle(p,q)}.$$
	
	We consider several  seminorms on the space $C^0( X^-\times \proj)$. 
	
	Define $v_\alpha^\Pp:C^0( X^-\times\proj)\to [0,\infty]$ by
	$$ v_\alpha^\Pp(\varphi):= \sup_{x^-\in  X^-}\sup_{\hat p\neq \hat q} \frac{\abs{\varphi(x^-,\hat p) - \varphi(x^-, \hat q)}}{\delta(\hat p, \hat q)^\alpha}. $$
	Define  $\valfaSigma:C^0( X^-\times\proj)\to [0,\infty]$ by
	$$ \valfaSigma(\varphi):= \sup_{\hat p\in\proj}\sup_{x^-\neq y^-} \frac{\abs{\varphi(x^-,\hat p) - \varphi(y^-, \hat p)}}{d(x^-, y^-)^\alpha}. $$
	Define also $v_\alpha:C^0( X^-\times\proj)\to [0,\infty]$ by
	$$ v_\alpha(\varphi):= v_\alpha^\Pp(\varphi)+\valfaSigma(\varphi). $$
	Finally denote by $C^\alpha( X^-\times \proj)$ the Banach space of  continuous functions $\varphi: X^-\to \C$ such that $v_\alpha(\varphi)<\infty$, i.e., $\alpha$-H\"older continuous observables,
	endowed with the norm
	$$ \norm{\varphi}_\alpha:= v_\alpha(\varphi)+\norm{\varphi}_\infty . $$

	\begin{proposition}
		\label{uniform convergence lemma}
		Given $x^-\in  X^-$ and $\hat p\in\proj $ for 
		$\muxminus$ almost every $x\in W^u_{{\rm loc}}(x^-)$
		$$\lim_{n\to \infty} \frac{1}{n}\,\log\norm{A^n(x)\,p}=L_1(A,\mu) .$$
		Moreover  
		$$\lim_{n\to \infty} \frac{1}{n}\,\int_{W^u_{loc}(x^-)} \log\norm{A^n(x)\,p}\, d\muxminus(x) =L_1(A,\mu) .$$
		with uniform convergence in 
		$(x^-,\hat p)\in  X^-\times\proj $.
	\end{proposition}
	
	\begin{proof}
		Fix some cilinder $[i]:=\{y^-\in  X^-\colon y_0=i\}$ and $x^-\in [i]$ such that $\muxminus$ almost every $x\in \Wuloc(x^-)$ belongs to the set $\mathscr{O}$ where Oseledet's theorem holds and let $\mathscr{O}_{x^-}:= \mathscr{O}\cap \Wuloc(x^-)$.

		Given  $y^-\in [i]$, for each $x\in W^u_{loc}(x^-)$ 
		define $h^s_{y^-}(x)$ as the unique point of  $W^s_{loc}(x)\cap W^u_{loc}(y^-)$, define 
		$\mathscr{O}_{y^-}:=h^s_{y^-}(\mathscr{O}_{x^-})$. Because  $\mu$ has a product structure we get that $\mu^u_{y^-}(\mathscr{O}_{y^-})=1$.

		By Proposition~\ref{constant oseledets},
		$\mathscr{O}_{y^-}\subseteq \mathscr{O}^s\cap \Wuloc(y^-)$
		and   $E^s(y) =H^s_{x,y}E^s(x)$ for every $y\in \mathscr{O}_{y^-}$.

		Let $\grass(d-1)$ be the Grassmanian space of $d-1$ dimensional subspaces of $\R^d$. Given   $\hat p\in \proj$ consider the hyperplane section defined by  $V_{\hat p} :=\{E\in \grass(d-1)\colon p\in E\}$. 
		
		Define a family of measures on $\grass(d-1)$ as $m^s_{y^-}:=\int \delta_{E^s}d\muyminus$. Next define a measure $m^s$ on $ X^-\times \grass(d-1)$ by $m^s:=\int_{ X^-}m^s_{y^-} d\mu^-(y^-)$, this is the projection of the $s$-state defined on $X\times \grass(d-1)$ by the disintegration $\delta_{E^s_y}$. By \cite[Proposition~4.4]{AvV1}, applied to $s$-states instead of $u$-states, we get that $y^-\mapsto m^s_{y^-}$ is continuous. Moreover, using the pinching and twisting assumption by \cite[Proposition~5.1]{AvV1} every hyperplane section $V$ has $m^s_{y^-}(V)=0$.

		Now take any $\hat p\in \proj$, by the previous observation we have that $m^s_{y^-}(V_{\hat p})=0$, which is equivalent to 
		$$
		\muyminus \{y\in W^u_{loc}(y^-)\colon p\notin E^s_{y}\}=1,
		$$
		so we get that for $\muyminus$ almost every $y$, $\lim_{n\to \infty} \frac{1}{n}\,\log\norm{A^n(y)\,p}=L_1(A,\mu)$.

		\bigskip
		
		Second step: Uniform convergence.
		We will make use of the co-norm of a matrix $g\in \GL$, defined by $m(g):=\min_{\norm{x}=1} \norm{g\, x}$. This quantity  can  also be characterized by $m(g)=\norm{g^{-1}}^{-1}$.

		The proof goes by contradiction.
		Assume there are sequences, which we can always assume to be convergent, $\hat p_n\to \hat p$ in $\proj$ and $x^-_n\to x^-$ in $ X^-$ such that
		$$\limsup_{n\to \infty} \frac{1}{n}\,\int_{W^u_{{\rm loc}}(x^-_n)} \log\norm{A^n(x)\,p_n}\, d\mu^u_{x_n^-}(x) <L_1(A,\mu) .$$

		For each $x\in W^u_{{\rm loc}}(x^-)$ let $h_n(x) =h^s_{x_n^{-}} (x)$ be the intersection of $W^u_{{\rm loc}}(x^-_n)$ with $W^s_{{\rm loc}}(x)$, so we get
		$$A^n(h_n(x))=H^s_{T^n(x),T^n(h_n(x))} A^n(x) H^s_{h_n(x),x}.$$
		Then we have 
		\begin{align}\label{eq.norm.ineq}
		& \norm{A^n(x)H^s_{h_n(x),x}\,p_n} \, m(H^s_{T^n(x),T^n(h_n(x))}) \leq  \norm{A^n(h_n(x))\,p_n}\\
		&\qquad \qquad \qquad  \qquad  \leq \norm{A^n(x)H^s_{h_n(x),x}\,p_n} \, \norm{H^s_{T^n(x),T^n(h_n(x))}}
		\nonumber 
		\end{align}

		Observe that $m(H^s_{T^n(x),T^n(h_n(x))})$ and 
		$\norm{H^s_{T^n(x),T^n(h_n(x))}}$ converge uniformly to $1$ when $n\to \infty$, so for $n$ large enough and for every $x\in W^u_{{\rm loc}}(x^-)$ we get
		$$
		\frac{1}{2}\norm{A^n(x)H^s_{h_n(x),x}\,p_n}\leq \norm{A^n(h_n(x))\,p_n}\leq 2\norm{A^n(x)H^s_{h_n(x),x}\,p_n}.
		$$

		Let $Jh_n$ be the Jacobian of $h_n$ from the measure $\muxminus$ to $\muxnminus$. By the local product  assumption, 
		$\mu=\rho \, (\mu^-\times \mu^+)$ and so 
		$\muxminus =\rho(x^-,\cdot)\mu^+$. Then as $\rho$ is continuous we get that $Jh_n$ converges uniformly to $1$ when $n\to \infty$.
		
		By \eqref{eq.norm.ineq} we have that 
		$$\begin{aligned}
		&\limsup_{n\to \infty} \frac{1}{n}\,\int_{W^u_{{\rm loc}}(x^-_n)} \log\norm{A^n(x)\,p_n}\, d\muxnminus(x) \\
		& \qquad = \limsup_{n\to \infty} \frac{1}{n}\,\int_{W^u_{{\rm loc}}(x^-)} \log\norm{A^n(x)H^s_{h_n(x),x}\,p_n}\, Jh_n(x)d\muxminus(x).
		\end{aligned}$$

		We have that $H^s_{h_n(x),x}\,p_n$ converges to $p$ uniformly on $x$, also we have that $\muxminus (\{x \colon  p\in E^s_x\})=0$ so for any $\delta>0$ we can find $n_0$ and $S_\delta\subset W^u_{{\rm loc}}(x^-)$, with $\muxminus(S_\delta)>1-\delta$, such that $\dist(H^s_{h_n(x),x}\,p_n,E^s_x)>\delta$ for every $x\in S_\delta$ and $n\geq n_0$.
		
		For any $x \in S_\delta \cap \mathscr{O}_{x^-}$ and any $\hat v \in \proj$ such that $\dist(v,E^s_x)>\delta$ we have that $\lim \frac{1}{n}\log \norm{A^n(x)v}=L_1(A,\mu)$, moreover this convergence is uniform  on the set $\{ \hat v\in \proj \colon  \dist(v,E^s_x)>\delta\}$.
		
		So we conclude that 
		$$
		\begin{aligned}
		\limsup_{n\to \infty} \frac{1}{n}\,\int_{W^u_{{\rm loc}}(x^-)} \log\norm{A^n(x)H^s_{h_n(x),x}\,p_n}\, Jh_n(x)d\muxminus(x)\geq& \\
		L_1(A,\mu)(1-\delta)-\log \norm{A} \delta&.
		\end{aligned}
		$$
		As $\delta$ can be taken arbitrarily small we get a contradiction.
	\end{proof}

	\begin{proposition}  \label{omega limit}
		Let $a=T^q(a)$ be a periodic point of $T$ such that $A^q(a)$ has eigenvalues with multiplicity one and distinct absolute values. Given $z\in W^s(a)$ and
		$\hat v\in\proj$, the sequence $\hat A^{n  q} (z) \, \hat v $ converges to an eigen-direction of $A^q(a)$ in $\proj$.
	\end{proposition}

	\begin{proof} 
		Choosing $z'=T^{l q}(z)\in \Wsloc(a)$ for some $l\geq 0$, by characterization of the stable holonomies in Section~\ref{statements} we have
		$$ H^s_{T^{(n-l)q}(z'),a }\, A^{(n-l) q}(z')= A^{(n-l) q}(a)\, H^s_{z',a}. $$ 
		which implies that
		$$ H^s_{T^{n q}(z),a }\,  A^{n q}(z)= A^{q}(a)^{n-l}\,  H^s_{z',a}\, A^{l q}(z) . $$
		Thus, since $H^s_{T^{n q}(z),a }\to I$, setting  $\hat v':= \hat H^s_{z',a}\, \hat A^{l q}(z) \, \hat v$,
		$$ \dist( \hat A^{n q}(z)\, \hat v, \hat A^q(a)^{n-l}\, \hat v') \to 0 \quad \text{ in } \; \proj .$$
		Finally, because $A^{q}(a)$ has eigenvalues with multiplicity one and distinct absolute values, the sequence $\hat A^q(a)^{n-l}\, \hat v'$   converges to the eigen-direction associated with the eigenvalue of largest absolute value in the spectral decomposition of a non-zero vector $v'$ aligned with $\hat v'$.
	\end{proof}

	\begin{proposition}
		\label{other homoclinics}
		If  $A:X \to\GL$  satisfies the pinching and twisting (see Definition~\ref{pinching and twisting}) for some homoclinic points $z\in\Wuloc(a)$ and  $z'=T^l z\in \Wsloc(a)$ then it also satisfies this condition with $z_k=T^{-k q} z$ and $z'_m=T^{l+m q} z$  for all  $k, m\in\N$.
	\end{proposition}

	\begin{proof}
		Using the holonomy relations in Section~\ref{statements}	
		\begin{align*}  
		& A^{m q}(y)\,H^s_{x,y} = H^s_{T^{m q}x, T^{m q}y} \, A^{m q}(x)\quad \text{ and  }\\
		& H^u_{x, y}\, A^{k q}(T^{-k q} x)  = A^{k q}(T^{-k q} y) \, H^u_{T^{-k q}x, T^{-k q}y} ,
		\end{align*}
		we obtain  
		\begin{equation}
		\label{other transition maps}
		\psi_{a, z_k, z_m'} = A^{m q}(a)\, \psi_{a, z, z'}\, A^{k q}(a) . 
		\end{equation}
		The conclusion follows because the subspaces  $\lspan\{e_j;j\in J\}$  in item (t) of Definition~\ref{pinching and twisting}  are $A^q(a)$-invariant.
	\end{proof}

	\begin{proposition}
		\label{Omegap dense in supp}
		For any $p\geq 1$ and  $(x^-,\hat v)\in  X^- \times \proj$, 
		$$\supp(m^-) \subseteq  \Omega_{K^p}(x^-, \hat v)  .$$
	\end{proposition}

	\begin{proof}
		By continuity it is enough to check this identity for a dense set of pairs $(x^-,\hat v)\in \supp(m^-)$. Hence we can assume that $x^-$ is an admissible sequence.	
		Recalling~\eqref{def Omega K}, the set $\Omega_{K^p}(x^-,\hat v)$ is the topological closure of the  set of all pairs 
		$$ ((x^-, i_0,\ldots, i_{n-1} ), \hat A_n (x^-, i_0,\ldots, i_{n-1} )\, \hat v )$$
		where $n$ is a multiple of $p$, $i_1,\ldots, i_{n-1}\in \{1,\ldots, \ell\}$,
		$(x^-, i_0,\ldots, i_{n-1} )$ is an admissible sequence and
		$$ A_n (x^-, i_0,\ldots, i_{n-1} ) :=A (x^-, i_0,\ldots, i_{n-1} )\, \ldots \, A(x^-, i_0)\, A(x^-) . $$
		Consider a periodic point $a=T^q(a)$ and associated 
		homoclinic  points $z\in W^s(a)\cap W^u(a)$ and  $z'=T^l(z)$  satisfying the pinching and twisting condition of  Definition~\ref{pinching and twisting}. 
		
		We break the proof in three steps:

		{\em Step 1:} \quad   $(y^-, \hat w)\in \Omega_{K^p}(x^-, \hat v)$ \; $\Rightarrow$\; 
		$ \Omega_{K^p}(y^-, \hat w)\subseteq  \Omega_{K^p}(x^-, \hat v)$.
		
		\smallskip

		This follows by continuity of the map $x^-\mapsto  A_n (x^-, i_0,\ldots, i_{n-1} )$.

		\bigskip
		
		For the next step, let $\{e_1,\ldots, e_d\}$ be an eigen-basis of  $A^q(a)$ respectively associated with eigenvalues $\lambda_1,\ldots, \lambda_d$ ordered in a way that $\abs{\lambda_i}>\ldots >\abs{\lambda_1}$.
		
		\smallskip
		
		{\em Step 2:} \quad $(a^-, \hat e_1)\in \Omega_{K^p}(x^-, \hat v)$.
		
		\smallskip
		
		Take $x\in W^s(a)$ admissible such that $x_j=x^-_j$, $\forall j\geq 1$. By Proposition~\ref{omega limit} 
		$\hat A^{n p q}(x)\, \hat v$ converges to $\hat e_i$
		for some $i=1,\ldots, d$. Since
		$$(P_-(T^{n p q}(x)), \hat A^{n p q}(x)\,\hat v)\in \Omega_{K^p}(x^-, \hat v), \quad  \forall n\in\N $$
		taking the limit as $n\to\infty$ we get $(a^-,\hat e_i)\in \Omega_{K^p}(x^-, \hat v)$. By Step 1 it is now enough to prove that
		$(a^-,\hat e_1)\in \Omega_{K^p}(a^-, \hat e_i)$.
		Because of Proposition~\ref{other homoclinics} we can replace $z$ by another point in the same orbit, still satisfying the pinching and twisting, and such that $z^-=a^-$. We can also assume that
		$l=k p q$, i.e., $z'=T^{k p q}(z)$ for some $k\geq 1$. By the twisting condition $\psi_{a,z,z'}(e_i)$ is transversal to
		$\lspan\{e_2,\ldots, e_d\}$.
		Defining $z_n:=T^{n p q+ k p q}(z)=T^{n p q}(z')$
		we have
		$$( z_n^-, \hat A^{(n+k) p q}(z)\,\hat e_i)\in \Omega_{K^p}(a^-, \hat e_i), \quad  \forall n\in\N .$$
		Finally using that $H^s_{z_n,a}\to I$, relation~\eqref{other transition maps}  and that the spectral decomposition of $\psi_{a,z,z'}(e_i)$ has a non zero component along $e_1$,
		\begin{align*}
		\lim_{n\to \infty} \hat A^{(n+k) p q}(z)\,\hat e_i &=
		\lim_{n\to \infty} \hat \psi_{a,z,z_n}(\hat e_i)\\
		&=
		\lim_{n\to \infty} \hat A(a)^{n p q}\,\hat \psi_{a,z,z'}(\hat e_i) = \hat e_1 ,
		\end{align*}
		which proves that $(a^-,\hat e_1)\in \Omega_{K^p}(a^-, \hat e_i)$.
		
		\bigskip

		{\em Step 3:} \quad $(y^-, \hat e^u(y^-))\in \Omega_{K^p}(a^-, \hat e_1)$, where $y^-$ is a $u$-regular admissible sequence in $ W^u(a)$, and    $\hat e^u(y^-)$ represents the Oseledets unstable direction $E^u(y^-)$ associated to the largest Lyapunov exponent.

		\smallskip

		Take $y\in X$ such that $P_-(y)=y^-$.
		Because $y\in W^u(a)$, we have $P_-(T^{-npq}(y))=a^-$, for all large $n\in\N$. Therefore, using the  invariance of
		the Oseledets direction $\hat e^u$,  
		\begin{align*}
		(y^-, \hat e^u(y^-)) \in \; & \Omega_{K^p}(P_-(T^{-npq}(y)), \hat e^u(P_-(T^{-npq}(y)))) \\
		&= \Omega_{K^p}(a^-, \hat e^u(a^-)) = \Omega_{K^p}(a^-, \hat e_1) .
		\end{align*}

		\bigskip

		To finish combine the conclusions of steps 2 and 3,
		and use the fact proved in Step 1 to derive
		that $(y^-, \hat e^u(y^-)) \in  \Omega_{K^p}(x^-, \hat v)$ for every  $u$-regular admissible sequence  $y^-$ in $ W^u(a)$. Since these points are dense in $\supp(m^-)$ it follows that
		$\supp(m^-)\subseteq  \Omega_{K^p}(x^-, \hat v)$.
	\end{proof}

	\begin{definition}
		\label{def kappa_alpha(An,mu)}
		Define for each $0<\alpha\leq 1$ and $n\in\N$,
		$$ \kappa_\alpha(A^n,\mu):= \sup_{x^-\in  X^-} \sup_{\hat v\neq \hat v'}
		\sum_{i_1=1}^\ell \cdots \sum_{i_n=1}^\ell  p_{i_1}\,\cdots p_{i_n}\, \left(  \frac{d( \hat A_{i_1, \ldots, i_n} \, \hat v, \hat A_{i_1, \ldots, i_n} \, \hat v') }{d(\hat v, \hat v')} \right)^\alpha   .
		$$
		where  $A_{i_1, \ldots, i_n}  := A(x^-_{n-1}) \cdots \hat A(x^-_0)$, $x^-_0:=x^-$, $x^-_{j}:=(x^-_{j-1}, i_j)$  and  $p_{i_j}:= p_{i_j}(x^-_{j-1})$ for $j=1,\ldots, n$.
	\end{definition}

	\begin{lemma}
		For any $\varphi\in C^0( X^-\times\proj)$,
		$$ v_\alpha^\Pp(\Qop^n\varphi)\leq \kappa_\alpha(A^n,\mu)\, v_\alpha^\Pp(\varphi) . $$
	\end{lemma}
	
	\begin{proof}
		Using the above notation 
		\begin{align*}
		(\Qop^n\varphi)(x^-, \hat p) = \sum_{i_1=1}^\ell \cdots \sum_{i_n=1}^\ell  p_{i_1}\,\cdots p_{i_n}\, \varphi\left(x_n^-, \hat A(x^-_{n-1}) \cdots \hat A(x^-_0)\, \hat p\right)  
		\end{align*}
		and hence 
		\begin{align*}
		& | (\Qop^n\varphi)(x^-, \hat v) - (\Qop^n\varphi)(x^-, \hat v') |  \\
		&\qquad  \leq v_\alpha^\Pp(\varphi)\, \sum_{i_1=1}^\ell \cdots \sum_{i_n=1}^\ell  p_{i_1}\,\cdots p_{i_n}\,   d\left( \hat A_{i_1, \ldots, i_n} \, \hat v, \hat A_{i_1, \ldots, i_n} \, \hat v'\right)^\alpha  \\
		&\qquad  \leq v_\alpha^\Pp(\varphi)\, d(\hat v, \hat v')^\alpha \, \sum_{i_1=1}^\ell \cdots \sum_{i_n=1}^\ell  p_{i_1}\,\cdots p_{i_n}\, \left(  \frac{d( \hat A_{i_1, \ldots, i_n} \, \hat v, \hat A_{i_1, \ldots, i_n} \, \hat v') }{d(\hat v, \hat v')} \right)^\alpha  \\
		&\qquad \leq v^\Pp_\alpha(\varphi)\, \kappa_\alpha(A^n,\mu) \, d(\hat v, \hat v')^\alpha .
		\end{align*}
		Thus dividing by $d(\hat v, \hat v')^\alpha$ and taking the sup in $x^-\in X^-$ and $\hat v \neq \hat v'$, the inequality follows.
	\end{proof}

	\begin{lemma}
		For every $n,m\in\N$
		$$ \kappa_\alpha(A^{n+m},\mu)\leq \kappa_\alpha(A^n,\mu)\,\kappa_\alpha(A^{m},\mu). $$
	\end{lemma}

	\begin{proof}
		Straightforward (see~\cite[Lemma 5.6]{DK-book}).
	\end{proof}

	\begin{proposition}
		\label{kappa alpha}
		There exist $n\in\N$ large enough and $0<\alpha<1$ small enough such that
		$$ \kappa_\alpha(A^{n},\mu)<1. $$
	\end{proposition}

	\begin{proof}
		Using	Proposition~\ref{uniform convergence lemma}, make a straightforward adaptation of the proof of~\cite[Lemma 2]{BD}.
	\end{proof}

	\begin{remark}
		Fixing $n\in\N$ and $0<\alpha <1$, the measurement
		$\kappa_\alpha(A^{n},\mu)$ depends continuously on $A$.
		In particular, the previous condition defines an open set in the space of
		fiber-bunched H\"older continuous cocycles.
	\end{remark}
	
	\begin{corollary}
		\label{v alpha Proj varphi ineq}
		There is $0<\sigma <1$ and $C<\infty$ such that for  $\varphi\in C^0( X^-\times \proj)$,
		$$ v_\alpha^\Pp(\Qop^n\varphi)\leq C\,\sigma^n \,v_\alpha^\Pp(\varphi) . $$
	\end{corollary}

	\begin{proposition}
		\label{v alpha Sigma Qop varphi ineq }
		Choosing $0<\alpha\leq 1$ according to Propositions~\ref{p alpha Holder} and~\ref{kappa alpha}, for any $\varphi\in C^\alpha( X^-\times\proj)$ and $n\geq 1$,
		$$ \valfaSigma(\Qop^n \varphi) \leq 2^{-n\,\alpha}\, \valfaSigma(\varphi) + \frac{v_\alpha(p)}{1-2^{-\alpha}}\, \norm{\varphi}_\infty + n  C M \beta^{n-1} \, v_\alpha^\Pp(\varphi)  $$
		where $\beta=\max\{2^{-\alpha}, \sigma\}<1$ and $M=M(\norm{A}_\infty)<\infty$.
	\end{proposition}

	\begin{proof} For each $k\in\N$ define  $v_k(\varphi)$ to be the sup of all oscillations
		$\abs{\varphi(x^-, \hat p)-\varphi(y^-, \hat p)}$
		such that $x^-$ and $y^-$ have the same $k$ coordinates. Then as before
		$$ \valfaSigma(\varphi)=\sup\{2^{k\alpha}\, v_k(\varphi)\colon k\in\N\} .$$ 	
		Assume $x^-, y^-\in  X^-$ with $x^-_{j}=y^-_{j},\, \forall\, j \geq -k $. Then for any $i\in \Sigma$ the first coordinates of the sequences
		$(x^-,i)$ and $(y^-,i)$ match up to order $k+1$. Hence
		
		\begin{align*}
		& \abs{(\Qop \varphi)(x^-,\hat p) - (\Qop \varphi)(y^-,\hat p) } 
		\\
		&\quad \leq  \sum_{i=1}^\ell \abs{ p_i(x^-)\, \varphi((x^-,i), \hat A(x^-) \hat p) -   p_i(y^-)\, \varphi((y^-,i), \hat A(y^-) \hat p) }\\
		&\quad \leq  \sum_{i=1}^\ell p_i(x^-)\, \abs{ \varphi((x^-,i), \hat A(x^-) \hat p) -    \varphi((y^-,i), \hat A(x^-) \hat p) } \\
		&\qquad + \sum_{i=1}^\ell p_i(x^-)\, \abs{ \varphi((y^-,i), \hat A(x^-) \hat p) -    \varphi((y^-,i), \hat A(y^-) \hat p) } \\
		&\qquad +  \sum_{i=1}^\ell \abs{p_i(x^-)-p_i(y^-)} \,  \abs{\varphi((y^-,i),\hat A(y^-) \hat p) }\\
		&\quad \leq v_{k+1}(\varphi)  +  v_\alpha^\Pp(\varphi)\, M \, d(x^-,y^-)^\alpha + v_k(p)\, \norm{\varphi}_\infty .
		\end{align*}
		Thus, taking the sup we get
		$$  v_k(\Qop \varphi) \leq v_{k+1}(\varphi) + M \,v_\alpha^\Pp(\varphi)\, 2^{-k\,\alpha} + \, v_k(p)\, \norm{\varphi}_\infty .$$
		Given $k\in \N$ we now have
		\begin{align*}
		2^{k\,\alpha}\, v_k(\Qop \varphi) &\leq 
		2^{k\,\alpha}\, v_{k+1}(\varphi) + M\, v_\alpha^\Pp(\varphi) + 2^{k\,\alpha}\, v_k(p)\, \norm{\varphi}_\infty\\
		&\leq 
		2^{-\alpha}\, v_{\alpha}(\varphi) + M\, v_\alpha^\Pp(\varphi) +  v_\alpha(p)\, \norm{\varphi}_\infty .
		\end{align*}
		Hence taking the sup
		$$ \valfaSigma(\Qop \varphi)\leq 
		2^{-\alpha}\, v_{\alpha}^\Sigma(\varphi) + M\, v_\alpha^\Pp(\varphi) + v_\alpha(p)\, \norm{\varphi}_\infty . $$
		Finally by induction we get
		\begin{align*}
		& \valfaSigma(\Qop^n \varphi) \leq 
		2^{-n\,\alpha}\, v_{\alpha}^\Sigma(\varphi) \\
		&\qquad + M (2^{-(n-1)\alpha}\, v_\alpha^\Pp(\varphi) + 2^{-(n-2)\alpha}\,v_\alpha^\Pp(\Qop\varphi)  + \cdots  +  2^{-0}\, v_\alpha^\Pp(\Qop^{n-1}\varphi)) \\
		& \qquad + (1+2^{-\alpha}+\cdots +2^{-(n-1)\alpha} )\, v_\alpha(p)\, \norm{\varphi}_\infty\\
		&\quad \leq 
		2^{-n\alpha}\, v_{\alpha}^\Sigma(\varphi) +  \sum_{i=0}^{n-1} M  \, 2^{-(n-1-i)\alpha} C\,\sigma^i\,v_\alpha^\Pp(\varphi) + \frac{v_\alpha(p)}{1-2^{-\alpha}}\, \norm{\varphi}_\infty\\
		&\quad \leq 
		2^{-n\alpha}\, v_{\alpha}^\Sigma(\varphi) +  n C M \beta^{n-1}\,v_\alpha^\Pp(\varphi) + \frac{v_\alpha(p)}{1-2^{-\alpha}}\, \norm{\varphi}_\infty  
		\end{align*}
		where $\beta=\max\{2^{-\alpha}, \sigma\}<1$.
	\end{proof}

	\begin{corollary}
		\label{Lasota-York type ineq}
		There exist $0<\sigma <1$ and $C<\infty$ such that for all $\varphi\in C^0( X^-\times \proj)$ and all sufficiently large $n$,
		$$ v_\alpha(\Qop^n\varphi)\leq \sigma^n \,v_\alpha(\varphi) + C\,\norm{\varphi}_\infty . $$
	\end{corollary}
	
	\begin{proof}
		Combine  Corollary~\ref{v alpha Proj varphi ineq} with Proposition~\ref{v alpha Sigma Qop varphi ineq }.
	\end{proof}

	Consider the probability measure 
	$\m\in \Prob( X^-\times\proj)$ which was proven to be 
	a stationary measure for 
	$\Qop\colon C^0( X^-\times\proj)\to  C^0( X^-\times\proj)$
	in Proposition~\ref{stationary measure}.  
	
	\begin{proposition} 
		\label{fibre string mixing}
		Take any  $0<\alpha\leq 1$ according to Propositions~\ref{p alpha Holder} and~\ref{kappa alpha}. Then the Markov operator $\Qop:C^\alpha(X^-\times\proj)\to C^\alpha(X^-\times\proj)$
		is strongly mixing. In other words  there exists   $0<\sigma_0<1$
		such that for any $\varphi\in C^\alpha(X^-\times\proj)$ and some $C_0=C_0(\norm{\varphi}_\alpha)<\infty$ we have 
		$$ \norm{ \Qop^n \varphi-\smallint \varphi\, dm^- }_\infty \leq C_0\, \sigma_0^n  \qquad \forall \, n\in\N .$$
	\end{proposition}
	
	\begin{proof}
		The strong mixing property follows from the abstract Theorem~\ref{base strong mixing}. Notice that Proposition~\ref{stationary measure} implies hypothesis (1), Proposition~\ref{Omegap dense in supp} guarantees hypothesis (2) while Corollary~\ref{Lasota-York type ineq} ensures hypothesis (3) of Theorem~\ref{base strong mixing}.
	\end{proof}


	\section{Continuity of the Lyapunov Exponents}
	\label{continuity}

	In this section we establish large deviation estimates of exponential
	type for fiber bunched cocycles  satisfying a pinching and twisting condition. These are then used to prove H\"older continuity of the top Lyapunov exponent on this space of cocycles.

	\bigskip

	Given $0<\alpha\leq 1$, denote by  $ C^\alpha_{\rm FB}(X, \GL)$  the space of fiber bunched $\alpha$-H\"older continuous cocycles  $A:X \to \GL$. Cocycles
	that factor through the projection $P_-\colon X \to X^-$
	are constant along local unstable sets $\Wuloc(x)$.
	We denote by $ C^\alpha_{\rm FB}(X^-, \GL)$  the subspace of all such cocycles. Both these spaces are endowed with the uniform distance $$d(A,B):= \norm{A-B}_\infty + \norm{A^{-1}-B^{-1}}_\infty .$$
	
	\begin{theorem} \label{fiber LDT Sigma-}
		Given  $A\in C^\alpha_{\rm FB}(X, \GL)$ satisfying the pinching and twisting condition (Definition~\ref{pinching and twisting})  and where $\alpha$ is taken according to the conclusions of Propositions \ref{p alpha Holder}  and~\ref{kappa alpha} as well as according to the exponent  $\beta$ in the conclusion of Proposition~\ref{prop.reduction},  there exists   $\mathscr{V}$ neighborhood of $A$ in $C^\alpha_{\rm FB}(X, \GL)$ and there exist
		$C=C(A)<\infty$ and $k=k(A)>0$   
		such that for all\, $0<\varepsilon<1$,  $B\in\mathscr{V}$ and  $n\in\N$,  
		$$  \mu \left\{ x\in X \colon \, \abs{\frac{1}{n}\,\log \norm{B^{n}(x)} - L_1(B) } > \varepsilon \,\right\} \leq C\,e^{ - {k\,\varepsilon^2 }\, n  } \;. $$
	\end{theorem}

	\begin{proof}
		Let $\Xspace$ be the subspace of cocycles in  $C^\alpha_{\rm FB}(X^-, \GL)$ 
		satisfying the pinching and twisting condition.
		For each $A\in \Xspace$ consider the SDS 
		$K:X^-\times\proj \to\Prob(X^-\times\proj)$ introduced in~\eqref{def Fiber kernel}. We will write $K=K_A$ to emphasize the dependence of $K$ on $A$. By Proposition~\ref{fibre string mixing}, the stationary measure $\m=\m_A$ in Proposition~\ref{stationary measure} is the unique such measure. 
		Consider  $\xi_A:X^-\times\proj \to \R$ defined by
		$$\xi_A(x^-, \hat p):= \log \norm{A(x^-)\, p} \ \text{ where } p \in \hat p \ \text{ is a unit vector} .$$
		Consider the  Banach algebra 
		$( C^\alpha(X^-\times\proj), \norm{\cdot}_\alpha))$, which belongs to a scale of Banach algebras satisfying assumptions (B1)-(B7) in~\cite[Section 5.2.1]{DK-book}.

		The class $\tilXsp$ of  observed Markov systems $(K_A, \m_A, \pm \xi_A)$ with $A\in \Xspace$  is a metric space, when endowed with the uniform distance between the underlying cocycles, that satisfies assumptions   (A1)-(A4) in~\cite[Section 5.2.1]{DK-book}. By construction assumption (A1) holds.
		Assumption (A2) follows from Proposition~\ref{fibre string mixing}.
		Assumption (A3) holds easily because $\xi_A\in C^\alpha(X^-\times\proj)$. Finally (A4) is easily checked, adapting the proof of~\cite[Lemma 5.10]{DK-book}.

		Given any $A\in \Xspace$,
		by~\cite[Theorem 5.4]{DK-book} (see also the proof of~\cite [Theorem 5.3]{DK-book}) there exists a neighbourhood $\mathscr{V}=B_\delta(A)\subseteq \Xspace$ and there are positive constants
		$C <\infty$, $k$ and $\varepsilon_0$, depending only on $A$  through $\norm{\xi_A}_\alpha$, such that for all\, $0<\varepsilon<\varepsilon_0$,  $B \in\mathscr{V}$ and  $n\in\N$,  
		\begin{equation}
		\label{Markov process fiber ldt}
		\Pp_B \left[ \; \abs{\frac{1}{n}\,\log \norm{B^n }  - L_1(B,\mu)  } > \varepsilon \,\right]  \leq C\,e^{ - {k\,\varepsilon^2 }\, n  } \;, 
		\end{equation}  
		where $\Pp_B\in\Prob( (X^-\times\proj )^\N )$ is any probability measure which makes the process $\{e_n:(X^-\times\proj )^\N\to X^-\times \proj\}_{n\geq 0}$, defined by
		$e_n \{ (x_j^-, \hat v_j) \}_{j\geq 0}$ $:= ( x_n^-, \hat v_n )$, a stationary Markov process with transition stochastic kernel $K_B$ and constant common distribution $\m_B$. 
		
		The constraint   $0<\varepsilon<\varepsilon_0$  can be relaxed to $0<\varepsilon<1$ replacing the constant $k$ by $k'=k\, \varepsilon_0^2$.

		Next consider the map $\pi:X\times\proj \to (X^-\times\proj )^\N$, $\pi( x, \hat v) :=(\{x_n^-\}_{n\geq 0}, \hat v_n)$, where $x=\{x_j\}_{j\in \Z}$,  $x_n^-=\{x_{n+j}\}_{j\leq 0} $ and $\hat v_n = \hat A^n(x)\, \hat v$ for all $n\in\N$. This projection makes  the following diagram commutative
		$$ \begin{CD} X\times\proj    @>\hat F>>  X\times\proj \\@V\pi VV        @VV\pi V\\ (X^-\times\proj )^\N     @>> \hat T >   (X^-\times\proj )^\N\end{CD} $$
		where $\hat F(x,\hat v):=(T x, \hat A(x)\, \hat v)$ and the bottom horizontal map $\hat T$ stands for the left shift map.
		Defining $\mathrm{e}\colon (X^-\times\proj )^\N \to X^-\times\proj$,
		$\mathrm{e} \{(x_n^-, \hat v_n)\}_{n\geq 0}:= (x_0^-,\hat v_0)$, the above process $\{e_n\}_{n\geq 0}$ is  
		$e_n= \mathrm{e}\circ \hat{T}^n$.
		
		Define a probability measure $m_B\in\Prob(X\times\proj)$
		which integrates every bounded measurable functions $\varphi\colon X\times\proj\to\R$ by  
		$$\int \varphi\, d m_B= \int_X \varphi(x, \hat e^u(x))\, d\mu(x) .$$
		
		\begin{proposition}\label{stationary Markov process}	If  $\Pp_B :=\pi_\ast m_B$ ~ then the process $\{e_n\colon (X^-\times\proj )^\N\to X^-\times \proj\}_{n\geq 0}$  is   Markov 
			with transition stochastic kernel $K_B$ and common distribution $\m_B$. 
		\end{proposition}
		
		\begin{proof}
			Since
			$$e_n\circ \pi = \mathrm{e}\circ \hat{T}^n \circ \pi =
			\mathrm{e} \circ \pi \circ \hat F^n =(P_-\times\mathrm{id})\circ \hat F^n , $$
			while $\m_B = (P_-\times\mathrm{id})_\ast m_B$, both processes $\{e_n\}_{n\geq 0}$ and $\{\tilde e_n :=e_n\circ \pi \}_{n\geq 0}$ are stationary with common distribution $\m_B$.
			
			Taking a cylinder
			$[i]:=\{ x^-\in X^-\colon  x^-_0=i\}$, with $1\leq i\leq \ell$, and a Borel set $E\subseteq \proj$ since 
			\begin{align*}
			(K_B)_{(x^-,\hat v)}([i]\times E) &= \sum_{j=1}^\ell p_j(x^-)\, \ind_{[i]\times E} ((x^-,j), \hat B(x^-)\, \hat v) )\\
			&= p_i(x^-) \, \ind_{E}(\hat B(x^-)\, \hat v) \\
			&= \mu_{x^-}( \Wuloc(x^-)\cap \Wuloc(x^-,i) )\, \ind_{E}(\hat B(x^-)\, \hat v) \\
			&= m_B\left( \, \tilde e_n \in [i]\times E \, \vert \, \tilde e_{n-1}=( x^-,\hat v) \, \right)  \\
			&= \Pp_B\left( \,  e_n \in [i]\times B \, \vert \,   e_{n-1}=(x^-,\hat v) \, \right) \\
			\end{align*}
			both processes $\{e_n\}_{n\geq 0}$ and $\{\tilde e_n \}_{n\geq 0}$ are Markov with transition stochastic kernel $K_B$.
		\end{proof}

		Thus, because the projection  $\pi:X\times\proj \to (X^-\times\proj)^\N$ preserves measure,  identifying
		any $B\in \mathscr{V} \subseteq \Xspace$  as a cocycle  $B\equiv B\circ P_-$ on $X$, 	 the large deviation estimate~\eqref{Markov process fiber ldt} holds for all $B\in \mathscr{V}$, $0<\varepsilon<1$ and $n\in\N$,
		\begin{equation}
		\label{Markov process fiber ldt 2}
		\mu  \left\{  x \in X  \colon \; \abs{\frac{1}{n}\, \log\norm{B^n (x)} - L_1(B,\mu)   } > \varepsilon \,\right\}  \leq C\,e^{ - {k\,\varepsilon^2 }\, n  } \;, 
		\end{equation} 
		where the deviation set  above 
		is the pre-image under $\pi$ of the deviation set $\Delta_n^-(B,\varepsilon)\subset (X^-\times\proj)^\N$ in~\eqref{Markov process fiber ldt}. 
		This completes the proof for  cocycles depending only on past coordinates.

		\bigskip

		As explained in Section~\ref{holonomy} (see Proposition~\ref{prop.reduction}) to each cocycle 
		$A\in C^\alpha_{\rm FB}(X, \GL)$ we associate a new coycle
		$A^s\in C^\beta_{\rm FB}(X^-, \GL)$, with $0<\beta<\alpha$,  which is conjugated to $A$ via holonomies. The holonomy reduction procedure of Avila-Viana $C^\alpha_{\rm FB}(X, \GL)\to
		C^\beta_{\rm FB}(X^-, \GL)$, $A\mapsto A^s$, is continuous.
		Since $A$ and $A^s$ are conjugated, if $A$ satisfies the pinching and twisting condition then so does $A^s$. By Theorem~\ref{fiber LDT Sigma-},
		there exists $\mathscr{V}^s$ neighborhood of $A^s$ in  $C^\beta_{\rm FB}(X^-, \GL)$ such that for all $B\in \mathscr{V}^s$,  $0<\varepsilon<1$ and $n\in\N$ 
		$$\mu  \left\{  x \in X  \colon \; \abs{\frac{1}{n}\, \log\norm{B^n (x)} - L_1(B,\mu)   } > \varepsilon \,\right\}  \leq C\,e^{ - {k\,\varepsilon^2 }\, n  }. $$
		By continuity of the reduction map $B\mapsto B^s$, there exists $\mathscr{V}$ neighborhood of $A$ such that $B^s\in\mathscr{V}^s$ for all $B\in \mathscr{V}$. 
		Since $B$ and $B^s$ are conjugated via holonomies, and the holonomies are uniformly bounded,  
		$$ \frac{1}{n}\, \log \norm{B^n(x)}=
		\frac{1}{n}\, \log \norm{(B^s)^n(x)} +\mathscr{O}(\frac{1}{n}) .$$
		Hence    large deviation estimates transfer over from $B^s$ to $B$.
	\end{proof}

	The following is a restatement of 
	Theorem~\ref{theorem.main}. 
	
	\begin{theorem} \label{LE continuity}
		Given  $A\in C^\alpha_{\rm FB}(X, \GL)$ satisfying the pinching and twisting condition (Definition~\ref{pinching and twisting}),  there is a neighbourhood  $\mathscr{V}$  of $A$ in $C^\alpha_{\rm FB}(X, \GL)$ and there exist constants
		$C=C(A)<\infty$ and $\theta=\theta(A)\in (0,1)$   
		such that for all\, $B_1, B_2\in\mathscr{V}$   
		$$ \abs{L_1(B_1,\mu)-L_1(B_2,\mu)}\leq C\, \norm{B_1-B_2}_\infty^\theta . $$
	\end{theorem}

	\begin{proof}
		By Theorems~\ref{base ldt} and~\ref{fiber LDT Sigma-} we can apply the ACT~\cite[Theorem 3.1]{DK-book}. The H\"older modulus of continuity
		follows from the exponential type of the previous large deviation estimates.
	\end{proof}

	\begin{remark}\label{LDP} In the setting of Theorem~\ref{fiber LDT Sigma-}, we can also establish a large deviations principle, that is, a more precise but asymptotic version of the finitary LDT estimate obtained in this theorem. 
		
		The argument uses the local G\"artner-Ellis Theorem (see~\cite[Chapter V Lemma 6.2]{BL}) applied to the sequence of random variables $X_n := \log \norm{A^n (x^-) p} - n L_1 (A)$, and it is an adaptation of results in~\cite[Section 5.2]{DK-book}. More specifically, the condition in the local G\"artner-Ellis Theorem is a consequence of~\cite[Proposition 5.13]{DK-book}, which is applicable in our setting, as already discussed in the proof of Theorem~\ref{fiber LDT Sigma-}. We leave the details of the proof to the interested reader and only state the result.
		
		For any $t$ close enough to $0$, consider the Markov-Laplace operator $\Qop_{A, t} := \Qop_A \circ D_{e^t \xi_A}$ on the  Banach algebra 
		$( C^\alpha(X^-\times\proj), \norm{\cdot}_\alpha))$, where $D_{e^t \xi_A}$ is the multiplication operator by $e^t \xi_A$.
		Denote by $\lambda_A (t)$ the maximal eigenvalue of the operator $\Qop_{A, t}$ and let $c_A^\ast (t)$ be the Legendre transform of the function $c_A (t) := \log \lambda_A (t)$.  
		
		Then for $\varepsilon$ small enough and for all unit vectors $p$ we have
		$$\lim_{n\to\infty} \, \frac{1}{n} \log \, \mu \left\{ x\in X \colon \, \abs{\frac{1}{n}\,\log \norm{A^{n}(x) p} - L_1(A) } > \varepsilon \,\right\}  = - c_A^\ast (\varepsilon) < 0 \, .$$
	\end{remark}


	\section{Central Limit Theorem}
	\label{clt section}
	
	In this section we prove the following Central Limit Theorem (CLT).

	\begin{theorem}\label{clt}
		Given  $A\in C^\alpha_{\rm FB}(X, \GL)$ satisfying the pinching and twisting condition (Definition~\ref{pinching and twisting}),  there exists $0<\sigma<\infty$ such that for every $v\in \R^d\setminus\{0\}$ and  $a\in\R$,
		$$
		\lim_{n\to +\infty}  \mu\left\{ x\in X \colon  \frac{\log\norm{A^n(x) \, v }-n\, L_1(A,\mu)}{\sqrt{n}} \leq a \,  \right\} =  \frac{1}{\sigma\sqrt{2\pi}}\int_{-\infty}^a e^{-\frac{t^2}{2\sigma^2}}dt
		$$
	\end{theorem}

	\begin{proof} If two cocycles $A$ and $B$ are conjugated via some bounded measurable matrix valued function then CLT can be transferred over from one to the other.
		
		Hence, since any cocycle $A\in C^\alpha_{\rm FB}(X, \GL)$ is conjugated to a cocycle $B\in C^\beta_{\rm FB}(X^-, \GL)$, for some
		$0<\beta<\alpha$, we can assume that $A$ does not depend on future coordinates, i.e.,  $A\in C^\alpha_{\rm FB}(X^-, \GL)$.
		
		We will use the following abstract CLT for stationary Markov processes of Gordin and Lif\v{s}ic~\cite{Go69},
		where $L^2(\Gamma,\mu)$ stands for the usual Hilbert space of square integrable observables with   norm
		$$ \norm{\varphi}_2:= \left(\int_\Gamma \abs{\varphi}^2\, d\mu  \right)^{1/2} <\infty .$$
		
		\begin{theorem}\label{gordin lifsic}
			Let $\{X_n\}_{n\geq 0}$ be a stationary Markov process with compact metric state space $\Gamma$, transition probability kernel $K$ and stationary measure $\mu\in \Prob(\Gamma)$. Let $\Qop:L^2(\Gamma,\mu)\to L^2(\Gamma, \mu)$ be the associated Markov operator defined by
			$$ (\Qop \varphi)(x):=\int_\Gamma \varphi(y)\, dK_x(y) \,.  $$
			Given $\psi\in L^2(\Gamma,\mu)$, if\; 
			$ \sum_{n=0}^\infty \norm{\Qop^n\psi}_2 <\infty   $ \;
			then  
			\begin{enumerate}
				\item[(i)] $\psi= \varphi-\Qop \varphi$ where $\varphi:=\sum_{n=0}^\infty \Qop^n\psi \in L^2(\Gamma,\mu)$ and $\int \varphi\, d\mu=0$;
				\item[(ii)] $\left(\mathrm{var}(\psi(X_1)+\cdots + \psi(X_n))\right)^{1/2}= \sigma\,\sqrt{n} + \Oscr(1)$ as $n\to \infty$, with $\sigma^2=\norm{\varphi}_2^2-\norm{\Qop \varphi}_2^2\geq 0$;	
				\item[(iii)] If $\sigma > 0$ then $n^{-1/2}\,(\psi(X_1)+\cdots + \psi(X_n))$ converges in distribution to $N(0,\sigma^2)$ as $n\to \infty$.
			\end{enumerate}
		\end{theorem}

		To apply this theorem, we consider the state space $\Gamma:=X^-\times \proj$ with the
		SDS $K=K_A$ defined in \eqref{def Fiber kernel} from a given fiber bunched H\"older continuous cocycle 
		$A:X^-\to \GL(d,\R)$. Let $\m\in \Prob(X^-\times\proj)$ be the $K$-stationary measure in Proposition~\ref{stationary measure}.
		Denote by $\Qop$ the corresponding Markov operator acting on
		$L^2(\Gamma, \m)$.

		The space $\Omega:=X\times\proj$ becomes a probability space when endowed with the probability measure $m\in\Prob(\Omega)$ characterized by
		$$ \int_\Omega \varphi\, dm = \int_X \varphi(x, \eu(x))\, d\mu(x) $$
		for bounded and measurable observables $\varphi:\Omega\to\R$.
		Consider the process $X_n:\Omega\to\Gamma$ defined by
		$$ X_n(x,\hat p):= \left( P_-(T^n x), \hat A^n(x)\,\hat p \right) . $$
		This is a stationary Markov process with transition probability kernel $K$ and common stationary distribution $\m$. See Proposition~\ref{stationary Markov process}.

		Finally consider the H\"older continuous observable $\psi:\Gamma\to\R$,
		$$\psi(x^-, \hat p):= \log \norm{A(x^-)\, p}-L_1(A\, \mu) , $$ 
		for which
		$$  \sum_{j=0}^{n-1} \psi(X_j(x,\hat p))= \log\norm{A^n(x) \, p }-n\, L_1(A,\mu) .$$
		
		By Proposition~\ref{fibre string mixing} the series
		$\sum_{n=0}^\infty \Qop^j \psi$ is absolutely convergent in the Banach
		algebra $(C^\alpha(\Gamma), \norm{\cdot}_\alpha)$ to a H\"older continuous function $\varphi:= \sum_{n=0}^\infty \Qop^j \psi$.
		Since $\norm{\cdot}_2\leq \norm{\cdot}_\infty\leq \norm{\cdot}_\alpha$, this implies 
		$\sum_{n=0}^\infty \norm{\Qop^n\psi}_2<\infty$,
		thus verifying the the assumption  of Theorem~\ref{gordin lifsic}.

		By item (iii) of this theorem 
		$$ \frac{\log \norm{A^n(x)\, p}-n\, L_1(A,\mu)}{\sqrt{n}} 
		= \frac{\sum_{j=0}^{n-1} \psi(X_j(x,\hat p)) }{\sqrt{n}}   $$
		converges in distribution to $N(0,\sigma^2)$ where
		$\sigma^2= \norm{\varphi}_2^2-\norm{\Qop\varphi}_2^2\geq 0$.

		We are left to prove that $\sigma^2>0$.
		Assume, by contradiction that
		$\sigma^2= \norm{\varphi}_2^2-\norm{\Qop\varphi}_2^2= 0$.
		Then
		\begin{align*}
		0 &\leq  \int  \left( (\Qop\varphi)(x) - \varphi(y) \right)^2\, dK_x(y)\, d\mu(x)\\ 
		&=  \int  \left\{  ((\Qop \varphi)(x))^2+ \varphi(y)^2
		-2\, \varphi(y)\, (\Qop \varphi)(x)\right\} \, dK_x(y) \, d\mu(x)\\
		&= \int  \left\{  \varphi(y)^2
		- ((\Qop\varphi)(x))^2 \right\} \, dK_x(y)\, d\mu(x)\\
		&= \int   \varphi(y)^2\, dK_x(y)\, d\mu(x) - \int 
		((\Qop\varphi)(x))^2   \,  d\mu(x)\\
		&= \norm{\varphi}_2^2- \norm{\Qop\varphi}_2^2= 0
		\end{align*}
		which implies that $\varphi(X_{j+1}(x,\hat p))=(\Qop\varphi)(X_j(x,\hat p))$ for all $j\geq 0$ and
		$m$-almost every $(x,\hat p)\in\Omega$.
		
		Hence for $m$-almost every $(x,\hat p)\in\Omega$,
		\begin{align*}
		\log \norm{A^n(x)\, p}-n\, L_1(A,\mu) &= \sum_{j=0}^{n-1}
		\psi(X_j(x,\hat p)) \\
		&= \sum_{j=0}^{n-1}
		\varphi(X_j(x,\hat p)) - (\Qop \varphi)(X_j(x,\hat p))\\
		&= \sum_{j=0}^{n-1}
		\varphi(X_j(x,\hat p)) -  \varphi(X_{j+1}(x,\hat p))\\
		&=  
		\varphi(X_0(x,\hat p)) -  \varphi(X_{n}(x,\hat p)) .
		\end{align*}
		Because $\varphi\colon\Gamma\to\R$ is continuous on a compact metric space $\Gamma$, the right hand side is uniformly bounded by some constant $M<\infty$, independent of $n$.
		Because the left hand side is also a continuous function on the compact space $\Omega$, this function is  uniformly bounded (in absolute value)  by the constant $M$ over $\supp(m)$.
		It follows that all periodic orbits inside  $\supp(m)$ must have Lyapunov exponents exactly equal to $L_1(A)$.

		We get a contradiction because the pinching and twisting condition, see 
		Definition~\ref{pinching and twisting}, implies the existence of periodic orbits in $\supp(m)$ with varying  Lyapunov exponents.
		To see this consider the fixed point $a=T a$ as well as the homoclinic points $z$ and $z'=T^l z$ in $W^s(a)\cap W^u(a)$ whose existence is prescribed  in Definition~\ref{pinching and twisting}. These orbits stay inside $\supp(m)$. For simplicity we have assumed $a$ is fixed instead of periodic.
		Let $\lambda=L_1(A)$ so that the matrix $A(a)$ has largest eigenvalue 
		$e^\lambda$ (in absolute value), and let $e_1$ be the associated unit  eigenvector.
		The transition map $\psi_{a,z,z'}$ has a limit when
		$l\to+\infty$ with $z,z'$ converging to $a$. Hence $\norm{\psi_{a,z,z'}}<e^{ l\, \lambda}$ for any large enough $l\in\N$.
		Next consider the closed pseudo-orbit of length $k+l+n$,
		$$  \underbrace{T^{-k} z, \ldots, T^{-1} z}_{k\;\text{times} }\, z, Tz, \ldots, z'=T^l z,
		\underbrace{T^{l+1}z,\ldots, T^{l+n} z}_{n\;\text{times} }$$
		which is shadowed by the orbit of a true periodic point $b_{k,n}\in\supp(m)$ near $T^{-k} z$. The definition and properties
		of holonomy imply the following matrix proximity relations
		$$ A^{k+l+n}(b_{k,n})\approx A^{k+l+n}(T^{-k} z) \approx A(a)^n\, \psi_{a,z,z'}\, A(a)^k $$
		with geometrically small errors in $n$ and $k$.
		The vector $e_1$ is a good approximation of the most expanding direction of the third of these matrices, which takes $e_1$ to a vector 
		of length $<e^{(k+l+n)\lambda}$   aligned with $e_1$ up to an angle of order $e^{-2n\lambda}$. The reasons for the loss in expansion are that $\norm{\psi_{a,z,z'}}<e^{ l\, \lambda}$ and $\psi_{a,z,z'} (e_1)\neq e_1$.
		This shows that the periodic point $b_{k,n}$
		has Lyapunov exponent $<\lambda$, therefore concluding the proof.
	\end{proof}


	\section{An application to Mathematical Physics}
	\label{example}

	In~\cite{CS95}  V. Chulaevski and T. Spencer, used a formalism introduced by Pastur and Figotin for the random case (Anderson-Bernoulli model)  to give an explicit positive lower bound on the Lyapunov exponent for Schr\"odinger cocycles with a smooth non-constant potential  and small coupling constant over an Anosov linear toral automorphism. Later J. Bourgain and W. Schlag established in~\cite{BS00}  Anderson localization for Schr\"odinger operators in this context.

	In this section we establish  (without providing any lower bound) the positivity and the H\"older continuity of the Lyapunov exponent for Schr\"odinger operators over uniformly hyperbolic  diffeomorphisms. The result holds for general H\"older continuous potential functions and small coupling constants.

	Let $f:M\to M$ be a diffeomorphism of  class $C^1$, $\Lambda\subseteq M$ be a topologically mixing 
	$f$-invariant basic set and $\mu\in\Prob(\Lambda)$ be the unique equilibrium state of some H\"older
	continuous potential on $\Lambda$. Given another 
	H\"older  continuous potential $v:\Lambda\to\R$,  consider the family of Schr\"odinger cocycles
	$S_{E,\lambda}:\Lambda\to\SL$
	$$  S_{E,\lambda}(x):=\begin{bmatrix}
	E- \lambda\, v(x) & - 1 \\ 1 & 0
	\end{bmatrix} .$$

	\begin{theorem}
		\label{prop:LE(example)>0}
		In this setting, if $\int v\, d\mu \neq 0$, given $\delta>0$ 
		there exists $\lambda_0>0$ such that for all $0<\abs{\lambda}<\lambda_0$ and $E\in\R$ with $\abs{E}<2-\delta$,  the cocycle
		$(f, S_{E,\lambda})$ satisfies the pinching and twisting properties. 
		
		Consequently, its Lyapunov exponent is positive and is a H\"older continuous function of $(E, \lambda)$. Moreover, the CLT and the large deviations principle hold for the iterates of this cocycle.
	\end{theorem}

	\begin{remark}
		Comparing with~\cite{BS00} where it is assumed that
		$\int v\, d\mu=0$ with $\delta <\abs{E}<2-\delta$, we do not require that $E\neq 0$ but assume instead that $\int v\, d\mu\neq 0$.
	\end{remark}
	
	\begin{proof}
		Assume that  $f:\Lambda\to \Lambda$ has a fixed point $f(p)=p$.
		Because $\abs{E}<2-\delta$, if $\lambda_0$ is small we have $\abs{\mathrm{tr}(S_{E,\lambda}(p))}=\abs{E-\lambda\, v(x)} <2$.
		Whence $S_{E,\lambda}(p)$ is an elliptic matrix with eigenvalues
		$\neq \pm 1$. If $f$ has no fixed points replace
		$f$ by some power $f^m$ with $m>1$. 
		Since $f$ is  topologically mixing, it admits periodic points $p=f^m(p)$ of any arbitrary given large period  $m\gg 1$. Making $\lambda_0$ small enough (depending on $m$), the matrix
		$S^m_{E,\lambda}(p)\approx \begin{bmatrix}
		E & -1 \\ 1 & 0
		\end{bmatrix}^m$ is elliptic.
		Choosing an appropriate $m$, $S_{E,\lambda}^m(p)$  has eigenvalues
		$\neq \pm 1$.
		From now on we assume $p=f(p)$.

		As in~\cite{BS00,CS95} we use Pastur Figotin formalism,
		writing $E=2\,\cos \kappa$ with $\kappa\in ]0,\pi[$, $v_n(x):= v(f^n x)$,
		$V_n(x):=-v_n(x)/\sin \kappa$ and  
		\begin{align*}
		& R_\theta :=\begin{bmatrix}
		\cos\theta  & -\sin \theta \\
		\sin\theta  & \cos \theta 
		\end{bmatrix} \\
		& N_\theta :=\begin{bmatrix}
		\sin \theta  & \cos\theta \\
		0  &  0
		\end{bmatrix} \\
		& M:=\begin{bmatrix}
		1 & -\cos \kappa \\
		0 & \sin \kappa 
		\end{bmatrix} \\
		& M^{-1}:=\frac{1}{\sin \kappa}\, \begin{bmatrix}
		\sin \kappa  &  \cos \kappa \\
		0 &  1
		\end{bmatrix} =
		\begin{bmatrix}
		1  &  \cot \kappa \\
		0 &  1
		\end{bmatrix}  .
		\end{align*} 
		Notice that 
		$$ R_\kappa = 
		M\, \begin{bmatrix}
		E &  -1 \\ 1 & 0
		\end{bmatrix}\, M^{-1}  . $$
		Then the conjugate cocycle
		$\tilde S_{E,\lambda}:=M\, S_{E,\lambda}\, M^{-1}$ is given by 
		$$ \tilde S_{E,\lambda}(x) := R_\kappa -\frac{\lambda\, v(x)}{\sin \kappa}  \, N_\kappa . $$

		\begin{lemma}
			\label{trace formula}
			For any fixed $n\in\N$, as $\lambda\to 0$,
			\begin{align}
			\nonumber
			\tilde S_{E,\lambda}^{n} & = R_{n\kappa} + \lambda\,    \sum_{j=1}^n V_j \, R_{(n-j) \kappa} \, N_{j \kappa}   +\mathscr{O}(\lambda^2) \, ,\\
			\label{tr formula}
			\mathrm{tr}\left[ \tilde S_{E,\lambda}^{n} \right] &= 
			2\, \cos(n\kappa) + \lambda\, \sin(n\kappa)\,  \sum_{j=1}^n V_j  
			+  \mathscr{O}(\lambda^2) \, .
			\end{align} 
		\end{lemma}

		\begin{proof}
			write  $Y_n:= \tilde S_{E,\lambda}^{n}$. Then $Y_0=I$ and 
			\begin{align*}
			Y_n &=  R_{n\kappa} + \lambda\, \sum_{j=1}^n V_j\, R_\kappa^{n-j}\, N_\kappa\, Y_{j-1} \\
			&=  R_{n\kappa} + \lambda\, \sum_{j=1}^n V_j\, R_{(n-j)\kappa}\, N_\kappa\, \left( R_{(j-1)\kappa} + \lambda\, \sum_{i=1}^j V_i\, R_\kappa^{j-i}\, N_\kappa\, Y_{i-1} \right)  \\
			&=  R_{n\kappa} + \lambda\, \sum_{j=1}^n V_j\, R_{(n-j)\kappa}\, N_\kappa\,R_{(j-1)\kappa} +   \lambda^2\, \sum_{1\leq i<j\leq n} V_j\,  V_i\, R_\kappa^{j-i}\,  N_\kappa^2\, Y_{i-1}  \\
			&=  R_{n\kappa} + \lambda\, \sum_{j=1}^n V_j\, R_{(n-j)\kappa}\, N_{j \kappa} +   \mathscr{O}(\lambda^2)  \\
			\end{align*}
			We have used above that $N_\theta\, R_{\theta'}=N_{\theta+\theta'}$,
			so that $N_\kappa\, R_{(j-1)\kappa} = N_{j \kappa}$.

			For the second statement just notice that $\mathrm{tr}(R_{n\kappa})=2\,\cos(n \kappa)$
			and \, $\mathrm{tr}(R_{(n-j)\kappa}\, N_{j \kappa})=\sin(n\kappa)$,
			because $\mathrm{tr}(R_{\theta'}\, N_\theta) =\sin(\theta+\theta')$.
		\end{proof}

		Later, to prove Lemma~\ref{elliptic, hyperbolic criterion} we need to assume that the eigenvalues of $S_{E,\lambda}(p)$ are not roots of unity of orders $1$, $2$, $3$, $4$, $6$ or $8$.
		
		\begin{lemma}
			Given $\delta>0$ there exists $\lambda_0>0$ such that for all $0<\abs{\lambda}<\lambda_0$ and $\abs{E}<2-\delta$ there is some periodic point $q$ of period $m$ such that $S_{E,\lambda}^{m}(q)$ is elliptic with eigenvalues of orders $\neq$ $1$, $2$, $3$, $4$, $6$ and $8$ and \,  $\abs{\mathrm{tr}(S_{E,\lambda}^{m}(q))}<2-\delta$.
		\end{lemma}

		\begin{proof}
			In our  context $\kappa\in ]0,\pi[$. If 
			$\kappa\notin \{\frac{\pi}{6},\frac{\pi}{4},\frac{\pi}{3},\frac{\pi}{2}, \frac{2\pi}{3},\frac{3\pi}{4},\frac{5\pi}{6} \} ,$
			or  equivalently if 
			$ E= 2\,\cos \kappa \notin \{-\sqrt{3},-\sqrt{2},-1,0,1,\sqrt{2},\sqrt{3}\}$
			we just need to keep  the fixed point $q=p$.
			Otherwise, in each of the four cases (i)
			$\kappa\in \{\frac{\pi}{6},  \frac{5\pi}{6} \}$,
			(ii) $\kappa\in \{\frac{\pi}{4}, \frac{3 \pi}{4}\}$,
			(iii)  $\kappa\in \{\frac{\pi}{3},  \frac{2\pi}{3} \}$
			and  (iv) $\kappa=\frac{\pi}{2}$ 
			we can use Lemma~\ref{trace formula} and Lemma~\ref{average along periodic orbits} below to find a matrix $S^m_{E,\lambda}(q)$ associated with a periodic point $q=f^m(q)$ such that
			$\sqrt{3} <\abs{ \mathrm{tr} ( S_{E,\lambda}^{m}(q) ) }<2-\delta $.

			In case (i) $\kappa\in \{\frac{\pi}{6},  \frac{5\pi}{6} \}$ there are infinitely many $m\in\N$ such that $2\,\cos(m\kappa)=\sqrt{3}$ and $\sin(m\kappa)=\pm \frac{1}{2}$. Choose $m$ with the appropriate remainder $r=(m\mod 6)$ so that  $\lambda\,\sin(m\kappa)$
			has the same sign as $\smallint V\,d\mu$.
			Then consider a sequence of periodic points $q=q_n$  provided by Lemma~\ref{average along periodic orbits} with periods $m=k_n$ such that $(m\mod 6)=r$.
			For these periodic points $q$ we have
			$$\lambda\, \sin(m\kappa)\, \sum_{j=1}^m V_j= \frac{\abs{\lambda}}{2}\, (m+o(m))\,\abs{\int V\, d\mu}$$
			and whence
			$$\mathrm{tr}\left[  S_{E,\lambda}^{m}(q) \right]  = 
			\sqrt{3} + \frac{\vert\lambda\vert}{2}\, (m+o(m)) \, \abs{\int V\, d\mu}
			+  \mathscr{O}(\lambda^2) .$$
			Taking $m$ large we can ensure this trace is above $\sqrt{3}$, while making $\lambda_0$ small enough
			we get many such $m$ where the trace is below $2-\delta$.

			In case (ii)  $\kappa\in \{\frac{\pi}{4}, \frac{3 \pi}{4}\}$ we find  a sequence of periodic points $q$ with periods $m$ such that
			$$\mathrm{tr}\left[  S_{E,\lambda}^{m}(q) \right]  = 
			\sqrt{2} + \frac{\vert\lambda\vert}{\sqrt{2}}\, (m+o(m)) \, \abs{\int V\, d\mu}
			+  \mathscr{O}(\lambda^2)   $$
			belongs to $]\sqrt{3}, 2-\delta[$.
			Similarly, in case (iii)  $\kappa\in \{\frac{\pi}{3},  \frac{2\pi}{3} \}$ we can find  a sequence of periodic points $q$ with periods $m$ such that
			$$\mathrm{tr}\left[  S_{E,\lambda}^{m}(q) \right]  = 
			-1 - \frac{\sqrt{3}\,\vert\lambda\vert}{2}\, (m+o(m)) \, \abs{\int V\, d\mu}
			+  \mathscr{O}(\lambda^2)  $$
			belongs to $]-2+\delta, -\sqrt{3}[$. 
			Finally in case (iv)  $\kappa=\frac{\pi}{2}$ there is a sequence of periodic points $q$ with odd periods $m$ such that
			$$\mathrm{tr}\left[  S_{E,\lambda}^{m}(q) \right]  = 
			\vert\lambda\vert \, (m+o(m)) \, \abs{\int V\, d\mu}
			+  \mathscr{O}(\lambda^2)  $$
			belongs to $]\sqrt{3}, 2-\delta [$. 
		\end{proof}

		\begin{lemma}
			\label{average along periodic orbits}
			There is 
			a sequence $\{p_n\}\subset \Lambda$ of periodic points with $\mathrm{per}(p_n)=k_n\to +\infty $ such that 
			$$ \lim_{n\to+\infty} \frac{1}{k_n}\,\sum_{j=0}^{k_n-1} v(f^j(p_n))= \int_\Lambda v\, d\mu .$$
			Moreover, given integers $d\in\N$ and $0\leq r <d$, the  $k_n$ can be chosen so that
			$k_n\mod{d}=r$.
		\end{lemma}

		\begin{proof}  
			Fix $\delta>0$ and a typical point $p\in \Lambda$ whose orbit is dense in $\Lambda$, with Birkhoff averages  $\frac{1}{n}\, \sum_{j=0}^{n-1}v(f^j p)$ converging to $\smallint v\, d\mu$. Next take $n_0\in\N$ such that 
			$\abs{ \frac{1}{n}\, \sum_{j=0}^{n-1}v(f^j p)-\smallint v\, d\mu }<\delta/2$  for all $n\geq n_0$.
			
			By uniform continuity  there exists $\varepsilon>0$ such that given $x,y\in \Lambda$ with $d(x,y)<\varepsilon$, $\abs{v(x)-v(y)}<\delta/2$.
			By the Shadowing Lemma~\cite[Proposition 8.20]{Shu87} there exists $\eta>0$ such that every periodic $\eta$-pseudo orbit of $f$ is $\varepsilon$-shadowed by some $f$-periodic orbit.
			Next choose a sequence of iterates $k_n\to +\infty$ such that $k_n\geq n_0$ and $d(p, f^{k_n}(p))<\eta$. By the Shadowing Lemma there exist periodic points $p_n$ with $\mathrm{per}(p_n)=k_n$ whose orbits are $\varepsilon$-near to the closed $\eta$-pseudo orbit
			$\{f^j(p)\}_{0\leq j \leq k_n-1}$.
			Putting these facts together
			\begin{align*}
			\abs{ \int v\, d\mu - \frac{1}{k_n}\,\sum_{j=0}^{k_n-1} v(f^j(p_n)) } &\leq 
			\abs{  \int v\, d\mu - \frac{1}{k_n}\,\sum_{j=0}^{k_n-1} v(f^j(p)) } \\
			& \qquad  + \frac{1}{k_n}\,\sum_{j=0}^{k_n-1} \abs{ v(f^j(p))- v(f^j(p_n)) } \\
			&\leq   \frac{\delta}{2} + \frac{1}{k_n}\, \sum_{j=0}^{k_n-1} \frac{\delta}{2} =\delta  
			\end{align*}
			which proves the  averages along the periodic orbits converge to the spatial average. Finally, given $d\in \N$, $r=0,1,\ldots, d-1$ and choosing the starting point $p$ to be $\eta$-close to a fixed point of $f$ we can take an
			$\eta$-pseudo orbit of the form $p,p,\ldots, p, f(p), f^2(p),\ldots, f^m(p)$, with $d(p,f^m(p))<\eta$, of length $k_n$ such that  $k_n \mod d=r$.
		\end{proof}

		The next couple of lemmas  hold in a context where   
		$f:\Lambda\to \Lambda$ is the same hyperbolic map but where  $A:\Lambda\to \SL$ is
		any H\"older continuous fiber-bunched cocycle.
		Given fixed points $p,p'\in \Lambda$ of $f$ 
		and a  heteroclinic point  $z\in \Wuloc(p)$ with $f^l(z)\in \Wsloc(p')$,  $l\in\N$,
		we define the {\em transition  map}  
		$$\psi_{p,z,f^l(z),p'}:\R^2\to\R^2\quad \text{ by }\quad \psi_{p,z,f^l(z),p'} := H^s_{f^l(z), p'}\, A^l(z)\, H^u_{p, z}.$$

		\begin{lemma}
			\label{monodromy}
			Let $p,p'\in \Lambda$ be fixed points  of $f$ and consider heteroclinic points
			$z\in \Wuloc(p)$ with $f^l(z)\in \Wsloc(p')$ and
			$z'\in \Wsloc(p)$ with $f^{-l'}(z')\in \Wuloc(p')$.
			If $p'$ is elliptic then 
			for every $m\in\N$ there are sequences of homoclinic points
			$q_n\in \Wuloc(p)$ converging to $z$ and $q_n'\in \Wsloc(p)$ converging to $z'$ such that $q_n'$ is a forward iterate of $q_n$ and 
			$$ \lim_{n\to \infty} \psi_{p,q_n, q_n', p}= \psi_{p', f^{-l'}(z'), z',p}\, A(p')^m\, \psi_{p,z,f^l(z),p'} .$$
		\end{lemma}
		
		\begin{proof}
			Take a sequence $y_n$ converging to $f^l(z)$ in the intersection of the curves $f^{-n}\left( \Wsloc(f^{-l'}(z'))\right)$ and $\Wuloc(f^{l}(z))$. Defining $q_n:=f^{-l}(y_n)$ and $q_n':=f^{n+l'}(y_n)$, the hyperpolicity at $p'$ implies
			that $f^l(q_n)=y_n\to f^l(z)$ and, since $l$ is fixed, $q_n\to z$.
			It also implies that $f^{-l'}(q_n')=f^n(y_n)\to f^{-l'}(z')$ and $q_n'\to z'$.
			See Figure~\ref{hetero}.	Next, given $m\in\N$, by Poincar\'e Recurrence Theorem,
			since $A(p')$ is elliptic there is a sequence of integers $t_n\to \infty$
			such that $A(p')^{t_n}\to A(p')^m$ and we reset $q_n:=q_{t_n}$, $q_n':=q_{t_n}'$, etc. Break $t_n$ as a sum
			$t_n=s_n+r_n$ of two divergent sequences of integers $r_n, s_n\to \infty$, for instance $s_n:=\lfloor t_n/2 \rfloor$ and $r_n:=t_n-s_n$.
			The hyperbolicity at $p'$ implies that
			$ f^{l+s_n}(q_n)=f^{s_n}(y_n)=f^{-l'-r_n}(q_n')$ converges to $p'$.
			Then we have
			\begin{align*}
			& \lim_{n\to \infty} \psi_{p,q_n,q_n', p}   = \lim_{n\to \infty} H^s_{q_n', p}\,  A^{l+t_n+l'}(q_n)\, H^u_{p, q_n}\\
			&\qquad  = \lim_{n\to \infty} H^s_{q_n', p}\, A^{r_n+l'}(f^{-l'-r_n}(q_n'))\, A^{l+s_n}(q_n)\, H^u_{p, q_n}\\
			&\qquad = \lim_{n\to \infty} H^s_{z', p}\, A^{r_n+l'}(f^{-l'-r_n}(z')) \, A(p')^{-r_n}\, A(p')^{  t_n} \, A(p')^{-s_n}\,  A^{l+s_n}(z)\, H^u_{p, z}\\
			&\qquad = \lim_{n\to \infty} \psi_{p', f^{-l'}(z'),z',p}\, A(p')^{t_n} \,
			\psi_{p,z,f^l(z),p'} \\
			&\qquad = \; \psi_{p', f^{-l'}(z'),z',p}\, A(p')^{m} \,
			\psi_{p,z,f^l(z),p'} .
			\end{align*}
			
			\begin{figure}
				\begin{center}
					\includegraphics[width=12cm]{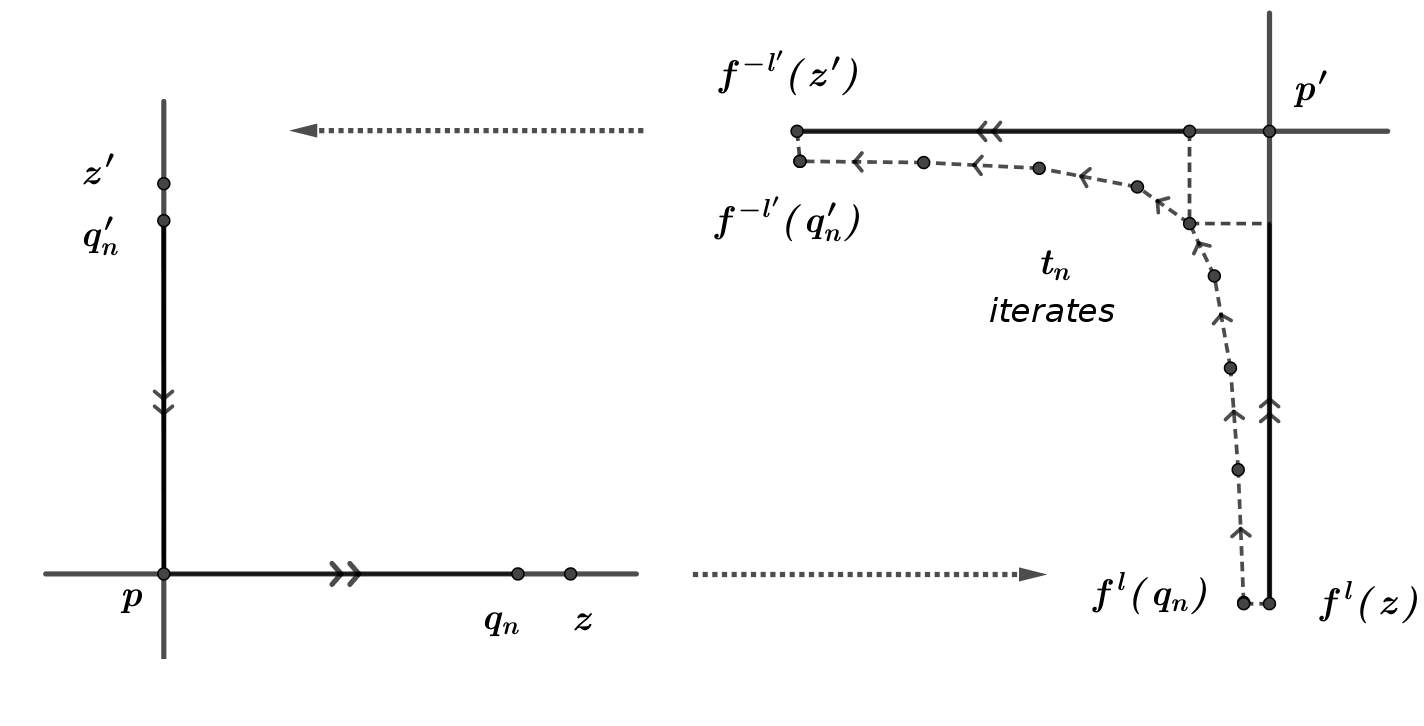}
				\end{center}
				\caption{Construction of homoclinic orbits shadowing the heteroclinic cycle
					$p\leadsto p' \leadsto p$. }
				\label{hetero}
			\end{figure}
			
			In the third step we use the continuity of the holonomies.
			Notice also that the hyperbolicity at $p'$ implies that the convergences $q_n \to z$,
			$f^{-l'-r_n}(q_n')\to p'$ and $f^{-l'-r_n}(z')\to p'$  are geometric with a rate which matches the strength of the base dynamics's hyperbolicity.
			Hence because of the fiber-bunched assumption,
			\begin{align*}
			& \lim_{n\to \infty} \norm{A^{l+s_n}(q_n)-A^{l+s_n}(z)}=0 \quad \text{ and }\\
			& \lim_{n\to \infty} \norm{A^{r_n+l'}(f^{-l'-r_n}(q_n'))-A^{r_n+l'}(f^{-l'-r_n}(z'))}=0 .
			\end{align*}

			Finally in the fourth step we use that
			\begin{align*}
			H^s_{f^l(z), p'} &= \lim_{n\to \infty} A(p')^{-s_n}\, A^{s_n}(f^l(z)) , \\
			H^u_{ p',f^{-l'}(z')} &= \lim_{n\to \infty} A^{r_n}(f^{-l'-r_n}(z'))\, A(p')^{-r_n}  .
			\end{align*}
			This concludes the proof.
		\end{proof}

		\begin{lemma}
			\label{elliptic, hyperbolic criterion}
			If $f$ admits two periodic points $p,p'\in \Lambda$, with periods $k$ and $k'$ respectively, such that $A^{k}(p)$ is hyperbolic and
			$A^{k'}(p')$ is elliptic with eigenvalues which are not roots of unity with orders $1$, $2$, $3$, $4$, $6$ or $8$,
			then  the cocycle satisfies pinching and twisting and hence has positive Lyapunov exponent.
		\end{lemma}
		
		\begin{proof}
			Working with the power $f^{k k'}$, we can assume that $k=k'=1$.
			Then $A(p)$ is hyperbolic and $A(p')$ is elliptic.
			The pinching condition follows from the hyperbolicity of  $A(p)$. We are left to prove the twisting condition for some transition map $\Phi$ associated with a homoclinic loop of $p$. Let $\{e_1,e_2\}$ be the eigenvector basis of $A(p)$  and $\hat e_1$, $\hat e_2$ be the corresponding projective points. We want to prove that  $\hat \Phi \hat\, e_1\neq \hat e_2$ and $\hat \Phi \,\hat e_2\neq \hat e_1$, where 
			$\hat \Phi$ stands for the projective action of $\Phi$.

			Take heteroclinic points $z$ and $z'$ such that $z\in \Wuloc(p)$ with $f^l(z)\in \Wsloc(p')$ and $z'\in \Wsloc(p)$ with
			$f^{-l'}(z')\in \Wuloc(p')$, where $l,l'\in\N$.
			Set  $\Psi_0:=\psi_{p,z,f^l(z),p'}$ and
			$\Psi_1:=\psi_{p',f^{-l'}(z'), z', p}$. By Lemma~\ref{monodromy} for every $m\in\N$,
			the $\SL$ matrix $\Phi_m:=\Psi_1\, A(p')^m\, \Psi_0$ can be approximated by transition maps of homoclinic loops of $p$. Let $\hat e_i^\ast:= \hat \Psi_0 \, \hat e_i$
			and $\hat e_i^\sharp:= \hat \Psi_1^{-1} \, \hat e_i$ for $i=1,2$. With this notation we  need
			to find $m\in\N$ such that $\{\hat e_1^\sharp, \hat e_2^\sharp\} \cap \hat A(p')^m \{\hat e_1^\ast ,\hat e_2^\ast\}  =\emptyset$. This suffices because any  transition map of a homoclinic loop of $p$ that approximates $\Phi_m$ well enough   satisfies the twisting condition.
			The assumption that the eigenvalues of $A(p')$ are not roots of unity with orders $1$, $2$, $3$, $4$, $6$ or $8$ 
			implies that the projective automorphism $\hat A(p')$ is either aperiodic or else it has period $\geq 5$. We claim there exists $1\leq m\leq 5$
			such that 
			\begin{equation}
			\label{disj}
			\{\hat e_1^\sharp, \hat e_2^\sharp\} \cap \hat A(p')^m \{\hat e_1^\ast ,\hat e_2^\ast\}  =\emptyset .
			\end{equation}
			Let  
			$\Oscr(\hat x) :=\{\hat A(p')^m\hat x\colon m\in\Z\}$ denote the $\hat A(p')$-orbit of a projective point $\hat x$. 
			If $\Oscr(\hat e_1^\ast)\neq \Oscr(\hat e_2^\ast)$ 
			these orbits are disjoint and it is not difficult to see that
			there are at least three  $m\in \{1,\ldots, 5\}$ such that~\eqref{disj} holds. Otherwise, if $\Oscr(\hat e_1^\ast) =\Oscr(\hat e_2^\ast)$ 
			and this orbit is disjoint from $\{\hat e_1^\sharp, \hat e_2^\sharp\}$
			then~\eqref{disj} holds for all $m\in\N$. If $\Oscr(\hat e_1^\ast) =\Oscr(\hat e_2^\ast)$ 
			and this orbit contains one element from $\{\hat e_1^\sharp, \hat e_2^\sharp\}$
			then~\eqref{disj} holds for three $m\in \{1,\ldots, 5\}$. 
			Finally, if $\Oscr(\hat e_1^\ast) =\Oscr(\hat e_2^\ast)$ 
			and this orbit contains  $\{\hat e_1^\sharp, \hat e_2^\sharp\}$
			then~\eqref{disj} holds for at least one $m\in \{1,\ldots, 5\}$. 
			This concludes the claim's proof.
			
			By Avila-Viana simplicity criterion, the positivity of the Lyapunov exponent follows.
		\end{proof}

		Since $\int v\, d\mu\neq 0$ without loss of generality we can assume that $\int v\, d\mu> 0$. Take 
		$0<\varepsilon<\frac{1}{2}\, \int v\, d\mu$. Next choose a Birkhoff generic point  $q\in \Lambda$ for the averages of $v$ which is also an $f$-recurrent point near the fixed point  $p$. More precisely choose $q\in \Lambda$ and  $n\in\N$  such that $d(p,q)<\varepsilon/4$,
		$d(f^{n}(q), q)<\varepsilon/4$ and 
		$\frac{1}{n}\, \sum_{j=1}^n v(T^j q)\geq \frac{4}{5}\, \int v\, d\mu$. 
		If $\varepsilon$ is small enough
		by the shadowing property there exists a periodic point $q_n=f^n(q_n)$  $\varepsilon/2$-near $q$. By the uniform continuity of $v$, if $\varepsilon$ is small enough
		we have
		$\frac{1}{n}\, \sum_{j=1}^n v(T^j q_n)\geq \frac{2}{3}\, \int v\, d\mu$.

		Thus by~\eqref{tr formula} provided  $\lambda_0$ is small enough,
		we have for all $\abs{\lambda}<\lambda_0$
		$$ \mathrm{tr}\left[ \tilde S_{E,\lambda}^{n} (q_n)\right]  \geq  
		2\, \cos(n\kappa) - n\, \lambda\, \frac{ \sin(n\kappa) }{\sin \kappa} \, \frac{1}{2}\, \int v\, d\mu  \gg 2 
		$$
		which implies that $ S_{E,\lambda}^{n} (q_n)$ is  hyperbolic.
		
		Finally,  Theorem~\ref{prop:LE(example)>0} follows
		by Lemma~\ref{elliptic, hyperbolic criterion}.
	\end{proof}

	\medskip

	\subsection*{Acknowledgments}
	
	The first author was supported  by Funda\c{c}\~{a}o para a Ci\^{e}ncia e a Tecnologia, under the projects: UID/MAT/04561/2013 and   PTDC/MAT-PUR/29126/2017.
The second author has been supported by the CNPq research grants 306369/2017-6 and 313777/2020-9. The third author was also supported by Instituto Serrapilheira, grant ``Jangada Din\^amica: Impulsionando Sistemas Din\^amicos na Regi\~ao Nordeste".

\bigskip

	\bibliographystyle{amsplain} 
	\providecommand{\bysame}{\leavevmode\hbox to3em{\hrulefill}\thinspace}
	\providecommand{\MR}{\relax\ifhmode\unskip\space\fi MR }
	\providecommand{\MRhref}[2]{%
		\href{http://www.ams.org/mathscinet-getitem?mr=#1}{#2}
	}
	\providecommand{\href}[2]{#2}

\end{document}